\documentclass[12pt]{amsart}

\usepackage{amsfonts}
\usepackage{amscd}
\usepackage{amssymb}

\allowdisplaybreaks
\usepackage{graphicx}
\usepackage{xcolor}
\usepackage{comment}
\usepackage{array}
\usepackage{tikz-cd}
\usepackage{cite}
\usepackage{enumerate,enumitem,hyperref,cleveref}
\hypersetup{colorlinks,linkcolor={blue},citecolor={blue},urlcolor={blue}}

\newtheorem{thm}{Theorem}[section]
\newtheorem{cor}[thm]{Corollary}
\newtheorem{lem}[thm]{Lemma}

\newtheorem{prop}[thm]{Proposition}
\newtheorem{defn}[thm]{Definition}
\theoremstyle{remark}
\newtheorem{rem}[thm]{Remark}

\numberwithin{equation}{section}
\newtheorem*{theoremA}{{\bf Theorem A}}
\newcommand{\N}{{\bf N}}
\usepackage{color}

\title[Weighted variational inequalities for the fractional Dunkl ...]
{Weighted variational inequalities for the fractional Dunkl heat semigroup}

\author[S. Mukherjee]{Suman Mukherjee}

\address{Department of Mathematics, Indian Institute of Technology Bombay, Powai, Mumbai--400076, India.}
\email{sumanmukherjee822@gmail.com}

\keywords{Fractional Dunkl heat semigroup, Variational inequalities, Dunkl--Calder\'on--Zygmund operators, Cotlar-type inequality.}
\subjclass[2020]{Primary: 42B20, 42B25. Secondary: 35R11, 47A35}

\begin{document}

\begin{abstract} 
We investigate the convergence properties of the family of operators
$$
\mathcal{T}_{\bf N} f(x)=\sum_{j=N_1}^{N_2}
v_j\Bigl(
e^{-a_{j+1}(-\Delta_k)^s}f(x)
-
e^{-a_j(-\Delta_k)^s}f(x)
\Bigr), \qquad x\in\mathbb{R}^d,
$$
where $\{e^{-t(-\Delta_k)^s}\}_{t>0}$ denotes the fractional heat semigroup generated by the Dunkl Laplacian $\Delta_k$. Here ${\bf N}=(N_1,N_2)\in\mathbb{Z}^2$, $N_1<N_2$, the coefficients $\{v_j\}_{j\in\mathbb{Z}}$ form a bounded sequence of real numbers, and $\{a_j\}_{j\in\mathbb{Z}}$ is a monotone increasing sequence of nonnegative reals. The primary objective of this work is to establish boundedness results for these differential transform operators on weighted $L^p(\mathbb{R}^d,d\mu_k)$ spaces as well as on Dunkl $\mathrm{BMO}(\mathbb{R}^d)$ spaces. We also establish analogous boundedness properties for the associated maximal operator
$$
\mathcal{T}^* f(x)=\sup_{\bf N} |\mathcal{T}_{\bf N} f(x)|
$$ and study the pointwise convergence of the corresponding series. In addition, we prove that, for compactly supported functions, the maximal differential transform operator $\mathcal{T}^*$ exhibits local behavior comparable to that of classical singular integral operators.
\end{abstract}
\maketitle
\tableofcontents

\section{Introduction and main results} \label{Sec-intro} 

Let $(X,\Sigma,m,\boldsymbol{\tau})$ be a dynamical system, where $(X,\Sigma,m)$ is a complete nonatomic probability space and $\boldsymbol{\tau}:X\to X$ is an invertible ergodic transformation preserving the measure $m$. Consider an increasing sequence of nonnegative integers $\{n_j\}_{j\in \mathbb{N}_0}$. For a measurable function $f$, define the corresponding ergodic averages by
$$
A_{j}f(x)=\frac1{n_j}\sum_{\ell=0}^{n_j-1}f(\boldsymbol{\tau}^\ell x).
$$
It is well known that no universal rate of convergence exists for ergodic averages. Discussions concerning this phenomenon may be found, for instance, in the books of Krengel and Petersen \cite{KrengelBook, PetersenBook}. Nevertheless, several quantitative estimates, including jump inequalities and upcrossing arguments, provide partial information regarding the mechanism of convergence. Observe that the series
$$
\sum_{j=1}^\infty \bigl(A_jf(x)-A_{j-1}f(x)\bigr)
$$
converges almost everywhere. Indeed, the partial sums telescope, leaving only the difference between the initial average and the final average, and the latter already converges. In contrast, absolute convergence may fail. More precisely, even for bounded functions one can have divergence of
$$
\sum_{j=1}^\infty
\left|
A_jf(x)-A_{j-1}f(x)
\right|.
$$
A nonnegative sequence $\{n_j\}_{j\in \mathbb{N}_0}$ is called $\ell$-lacunary, for some $\ell>1$, if $n_{j+1}\geq \ell n_j$ for all $j\in \mathbb{N}_0$. A result of Akcoglu, Jones, and Schwartz \cite{AkcogluVIPEAA} has shown that if $\{n_j\}_{j\in \mathbb{N}_0}$ is $\ell$-lacunary  for some $\ell>1$, and $\rho<2$, then there exist bounded functions for which the $\rho$-variation
$$
\left(
\sum_{j=1}^\infty
|A_jf(x)-A_{j-1}f(x)|^\rho
\right)^{1/\rho}
$$
diverges almost everywhere. Furthermore, in the classical case $n_j=j$, the above absolute series diverges for every nonconstant function. Hence the cancellation present in the telescoping sum plays an essential role.
Motivated by this observation, it is natural to investigate weighted difference series of the form
$$
\sum_{j=1}^\infty
v_j\bigl(
A_jf(x)-A_{j-1}f(x)
\bigr),
$$
where $\{v_j\}_{j\in \mathbb{N}_0}$ is a bounded sequence of real or complex scalars. Using Rademacher functions together with the boundedness of the square function
$$
S(f)(x)
=
\left(
\sum_{j=1}^\infty
|A_jf(x)-A_{j-1}f(x)|^2
\right)^{1/2}
$$
on $L^p$, $1<p<\infty$, one obtains convergence for almost every random choice of coefficients $v_j\in\{\pm1\}$. This naturally raises the question of whether convergence continues to hold for every bounded sequence $\{v_j\}_{j\in \mathbb{N}_0}$. In this direction Jones and Rosenblatt \cite{JonesDAET}
proved the following result.
\begin{theoremA}
Let $\{n_j\}_{j\in \mathbb{N}_0}$ be an $\ell$-lacunary sequence of nonnegative integers, where $\ell>1$; and let $\{v_j\}_{j\in \mathbb{N}_0}$ satisfy
$$
\|\{v_j\}\|_{\ell^\infty(\mathbb{N}_0)}\leq 1.
$$
Define
$$
\mathcal{T}_nf(x)
=
\sum_{j=1}^n
v_j\bigl(
A_jf(x)-A_{j-1}f(x)
\bigr).
$$
Then $\{\mathcal{T}_nf(x)\}_{n \in \mathbb{N}_0}$ converges almost everywhere for every $f\in L^p$, $1\leq p\leq\infty$. Moreover, if $1<p<\infty$, then $\mathcal{T}_nf$ converges in the $L^p$ norm.
\end{theoremA}

These developments underscore the deep interplay between ergodic theory, harmonic analysis, and probability theory, particularly in the study of maximal, variational, and square-function estimates for ergodic averages. Probabilistic methods, including randomization and martingale techniques, have proved to be powerful tools for establishing such estimates and for investigating the associated convergence phenomena. Motivated by these connections, a compelling direction is to study variation operators of the form
$$
\sum_{j=1}^n
v_j\bigl(
{\bf T}_{a_j}f(x)-{\bf T}_{a_{j-1}}f(x)
\bigr),$$
associated with various classes of semigroups $\{{\bf T}_t\}_{t>0}$, where the sequence $\{a_j\}_{j\in\mathbb N_0}$ satisfies suitable assumptions. The study of these operators has attracted considerable interest, with particular emphasis on establishing their boundedness and understanding the convergence behavior of the underlying semigroup averages. For instance, the $\rho$-variation of the heat semigroup in the Hermitian setting was studied in \cite{BetancorTPVOTHSITHS}. The boundedness of variation operators for the heat semigroups generated by Schr\"odinger operators was established in \cite{ChaoBODTFHSGBSO}. Weighted variation inequalities for differential operators and singular integrals were obtained in \cite{MaWVIFDOASI}. The boundedness of differential transforms for heat semigroups generated by the fractional Laplacian was considered in \cite{RenBODTFHSGBFL}, while the corresponding results for fractional heat semigroups generated by Schr\"odinger operators were proved in \cite{LiBODTFFHSGBSO}.
 \par Keeping the motivation arising from these developments in view, we now turn our attention to another framework in harmonic analysis. The theory developed by Charles F. Dunkl \cite{Dunkl} can be viewed as a nontrivial extension of classical Fourier analysis on Euclidean spaces, and over the years it has become an active area of research. A fundamental role in this theory is played by the Dunkl operators, which replace the ordinary directional derivatives appearing in the classical setting. 

\par An important second-order operator associated with this structure is the Dunkl Laplacian $\Delta_k$, defined by
$$
\Delta_k f(x)
=\Delta f(x)
+\sum_{\lambda\in R}k(\lambda)
\left(
\frac{\langle \nabla f(x),\lambda\rangle}{\langle \lambda,x\rangle}
-\frac{|\lambda|^2}{2}
\frac{f(x)-f(\sigma_\lambda x)}{\langle \lambda,x\rangle^2}
\right), \qquad x \in \mathbb{R}^d,
$$
where all symbols and conventions are introduced in Section \ref{Sec-prel}. This operator occupies a key position in the investigation of evolution equations and semigroup theory within the Dunkl setting. In particular, R\"osler considered in \cite{RoslerGHPATHEFDO} the Cauchy problem corresponding to the Dunkl heat equation
\begin{align*}
\begin{cases}
\partial_t u(x,t)-\Delta_k u(x,t)=0,
\qquad (x,t)\in \mathbb{R}^d\times (0,\infty),\\
u(x,0)=f(x).
\end{cases}
\end{align*}
It is known that the function $e^{t\Delta_k}f(x)$ provides a solution to the Dunkl heat equation, and the family $\{e^{t\Delta_k}\}_{t>0}$ forms the Dunkl heat semigroup. Motivated by the recent developments on variational inequalities associated with various semigroups, discussed earlier, the corresponding variational inequalities for this semigroup were studied in \cite{DasBODTFDHS}. More precisely, the authors considered the series
\begin{equation*}
\sum_{j\in \mathbb{Z}}
v_j
\Big(
e^{a_{j+1}\Delta_k}f(x)
-
e^{a_j\Delta_k}f(x)
\Big),
\end{equation*}
and established the boundedness of the partial sum operator on weighted $L^p(\mathbb{R}^d,d\mu_k)$ spaces and on Dunkl $\mathrm{BMO}(\mathbb{R}^d)$ spaces. They also established analogous boundedness results for the associated maximal operator and obtained pointwise convergence results for the related series.

\par Now, in analogy with the classical setting, one may also define fractional powers of the Dunkl Laplacian. Specifically, for each $s>0$, these are introduced through the Dunkl transform by
$$
\mathcal{F}_k\big((-\Delta_k)^s f\big)(\xi)
=
|\xi|^{2s}\mathcal{F}_k(f)(\xi).
$$
Consequently, when $0<s<1$, one obtains the semigroup
$\{e^{-t(-\Delta_k)^s}\}_{t\geq 0}$,
which is referred to as the fractional Dunkl heat semigroup. See Section \ref{subsec-fractional Dunkl heat} for details.
\par The fractional Dunkl heat semigroup has recently attracted considerable attention in Dunkl analysis. For instance, it has been employed in the study of sharp Dunkl multiplier theorems \cite{SumanOTBODM}, while the asymptotic behavior of solutions associated with the fractional Dunkl heat equation has been investigated in \cite{SumanLTBOFFHEAWTDL}. Moreover, as in the classical setting, this semigroup is expected to play an important role in the study of the linear part of solutions to several fluid equations arising in mathematical physics, such as the generalized Navier--Stokes equations, the quasi-geostrophic equations, and the magnetohydrodynamic equations associated with the Dunkl operators. In probability theory, it may also be used to describe certain classes of Markov processes with jumps.

\par Motivated by the discussion at the beginning, it is therefore natural to investigate the properties of the series
\begin{equation}\label{1st equation}
\sum_{j\in \mathbb{Z}}
v_j
\Big(
e^{-a_{j+1}(-\Delta_k)^s}f(x)
-
e^{-a_j(-\Delta_k)^s}f(x)
\Big),
\end{equation}
where $0<s<1$, the sequence $\{v_j\}_{j\in\mathbb{Z}}$ is bounded and real-valued, and $\{a_j\}_{j\in\mathbb{Z}}$ is an increasing sequence of positive real numbers. To study the behavior of the series \eqref{1st equation}, we consider the maximal operator defined by
$$\mathcal{T}^* f(x)= \sup\limits_{\N=(N_1, N_2)\in \mathbb{Z}^2} \big|\mathcal{T}_{\N} f(x)\big|,$$
where $\mathcal{T}_{\N}$'s are the partial sum operators given by 
$$\mathcal{T}_{\N}f(x)=\sum\limits_{j=N_1}^{N_2} v_j \Big(e^{-a_{j+1}(-\Delta_k)^s}f(x) - e^{-a_j(-\Delta_k)^s} f(x)\Big), \quad 0<s<1.$$
We will study the boundedness of the maximal operator in different weighted function spaces when the sequence $\{a_j\}_{j\in \mathbb{Z}}$ is a \emph{lacunary sequence}. Recall that a sequence $\{a_j\}_{j\in \mathbb{Z}}$ of nonnegative reals is called $\ell$-lacunary, for some $\ell>1$, if $a_{j+1}\geq \ell a_j$ for all $j\in \mathbb{Z}$. Our results regarding the weighted boundedness of $\mathcal{T}^*$ are as follows.

\begin{thm}\label{main thm for maximal op}
Let $d\geq 2,\ell>1$, and $\{a_j\}_{j\in \mathbb{Z}}$ be an $\ell$-lacunary sequence. Then the following are true.
\begin{enumerate}[label=(\roman*)]
    \item \label{T star p} If $1<p<\infty$ and $w$ is a $G$-invariant weight such that $w\in A^k_p$, then for any $f\in L^p(\mathbb{R}^d, w\,d\mu_k)$ we have the strong type inequality
    $$\Big(\int_{\mathbb{R}^d}(\mathcal{T}^*f(x))^p w(x)\, d\mu_k(x)  \Big)^{1/p}\leq C\, \Big(\int_{\mathbb{R}^d}|f(x)|^p w(x)\, d\mu_k(x)  \Big)^{1/p},$$
    where $C=C(d, k, G, s, p, w, \|\{v_j\}\|_{\ell^\infty(\mathbb{Z})}, \ell)$. 

    \item \label{T star 1} If  $w$ is a $G$-invariant weight such that $w\in A^k_1$, then for any $f\in L^1(\mathbb{R}^d, w\,d\mu_k)$ we have the weak type inequality
    $$\sup\limits_{t>0}\,t\Big(\int\limits_{\{\mathcal{T}^*f(x)>t\}} w(x)\, d\mu_k(x)  \Big)\leq C\, \int_{\mathbb{R}^d}|f(x)| w(x)\, d\mu_k(x)  ,$$
    where $C=C(d, k, G, s, w, \|\{v_j\}\|_{\ell^\infty(\mathbb{Z})}, \ell)$. 

    \item \label{T star l infty} For every $f\in L^\infty(\mathbb{R}^d,d\mu_k)$, either $\mathcal{T}^*f=\infty$ $\mu_k$-almost everywhere, or
$\mathcal{T}^*f<\infty$ $\mu_k$-almost everywhere and
    $$\|\mathcal{T}^*f\|_{\text{BMO}_k} \leq C\, \|f\|_{L^\infty(d\mu_k)},$$
     where $C=C(d, k, G, s, \|\{v_j\}\|_{\ell^\infty(\mathbb{Z})}, \ell)$. 

     \item \label{T star BMO} For every $f\in\mathrm{BMO}_{G,k}(\mathbb{R}^d)$,
either $\mathcal{T}^*f=\infty$ $\mu_k$-almost everywhere, or
$\mathcal{T}^*f<\infty$ $\mu_k$-almost everywhere and
    $$\|\mathcal{T}^*f\|_{\text{BMO}_k} \leq C\, \|f\|_{\text{BMO}_{G,k}},$$
     where $C=C(d, k, G, s, \|\{v_j\}\|_{\ell^\infty(\mathbb{Z})}, \ell)$. 
    \end{enumerate}
\end{thm}

\begin{rem}
In the Dunkl setting, there are two distinct types of BMO spaces: one associated with the usual balls, denoted by $\mathrm{BMO}_{k}(\mathbb{R}^{d})$, and the other associated with the orbits of balls, denoted by $\mathrm{BMO}_{G,\,k}(\mathbb{R}^{d})$. Their definitions and the relationship between them are explained in Section \ref{subsec-BMO sp}. We have established the boundedness of the maximal operator from 
$L^{\infty}(\mathbb{R}^{d}, d\mu_k)
$ to $\mathrm{BMO}_{k}(\mathbb{R}^{d}),$
and from $\mathrm{BMO}_{G,\,k}(\mathbb{R}^{d})$ to $\mathrm{BMO}_{k}(\mathbb{R}^{d})$. The boundedness in the other direction  and from 
$L^{\infty}(\mathbb{R}^{d}, d\mu_k)
$ to $\mathrm{BMO}_{G,\,k}(\mathbb{R}^{d}),$ remain unknown.
\end{rem}

\par Our results may be viewed as fractional analogues of those obtained in \cite[Theorems 4.2 and 4.3]{DasBODTFDHS} for the Dunkl heat semigroup. The main novelty of the present work stems from the fact that several pointwise estimates for the Dunkl heat kernel are already available and play a crucial role in \cite{DasBODTFDHS}, whereas no comparable estimates are known for the fractional Dunkl heat kernel. To overcome this difficulty, we first investigate certain properties of the Bochner subordination function (see Propositions \ref{Subordi func esti} and \ref{esti of deriv of subordinate func}) and, based on these results, establish several fundamental estimates for the fractional Dunkl heat kernel (see Proposition \ref{esti for fractional Dunkl heat kernel}). These estimates form the key ingredients in the proofs of our main results. In several instances, our arguments also yield stronger conclusions than those obtained in \cite{DasBODTFDHS}; see, for example, the pointwise convergence results in Corollary \ref{main thm point wise conv}. Furthermore, we show that the kernel estimates for the partial sum operator can be derived directly (see Proposition \ref{kernel of TN}), without passing through multiplier operator techniques as in \cite{DasBODTFDHS}. In addition, for the local growth estimates of the associated maximal operator, our results significantly improve the corresponding bounds in \cite[Theorem 5.2]{DasBODTFDHS}; see Theorem \ref{main thm local growth}.

As a corollary of the boundedness of the maximal operator, we obtain the following pointwise convergence result for the partial sum operators.

\begin{cor}\label{main thm point wise conv}
  Let $d\geq 2,\ell>1$, and $\{a_j\}_{j\in \mathbb{Z}}$ be an $\ell$-lacunary sequence. Then the following are true.
\begin{enumerate}[label=(\roman*)]
    \item \label{part 1 conv} If $1<p<\infty$ and $w$ is a $G$-invariant weight such that $w\in A^k_p$, then for any $f\in L^p(\mathbb{R}^d, w\,d\mu_k)$, $\mathcal{T}_{\N} f$ converges almost everywhere and in $L^p(\mathbb{R}^d, w\,d\mu_k)$ as $\N \to (-\infty, \infty)$.
    \item If  $w$ is a $G$-invariant weight such that $w\in A^k_1$, then for any $f\in L^1(\mathbb{R}^d, w\,d\mu_k)$, $\mathcal{T}_{\N} f$ converges almost everywhere and in $w\,d\mu_k$-measure as $\N \to (-\infty, \infty)$.
    \end{enumerate}
  \end{cor}
We next consider a classical example from harmonic analysis, obtained by taking
$f=\chi_{(0,1)}$ and letting $\mathcal{H}$ denote the Hilbert transform. In this case, one has
$$\frac{1}{r}\int_{-r}^{0} \mathcal{H}(f)(x),dx \sim \log \frac{e}{r},
\qquad r\to 0^{+}.$$
This logarithmic divergence describes the typical local growth near the origin for singular integral operators acting on bounded functions. The next theorem establishes that, for every bounded function $f$, the operator $\mathcal{T}^*f$ exhibits the same asymptotic order at the origin as that of a singular integral transform. We also mention that several related investigations concerning the local behavior of variation operators appear in \cite{BetancorTPVOTHSITHS}. Results in the one-dimensional setting for variation operators associated with convolution families were obtained in \cite{MaWVIFDOASI}. Furthermore, similar one-dimensional estimates for differential transforms corresponding to sequences of one-sided fractional Poisson-type operators were proved in \cite{ChaoBODTFOSFPTOS}.

\par However, in the Dunkl setting for the Dunkl heat semigroup, the authors in \cite{DasBODTFDHS} were not able to obtain the growth order $\big(\log \frac{2}{r}\big)^{1/p'}$; they obtained only the weaker order $r^{-d_k/p'}$. Here, for the fractional Dunkl heat semigroup, we improve their result and recover exactly the order $\big(\log \frac{2}{r}\big)^{1/p'}$. This is the main novelty of the result. Moreover, we mention that the same improvement can also be achieved for the Dunkl heat semigroup considered in their work.

\begin{thm}\label{main thm local growth}
  Let $d\geq 2,\ell>1, \{v_j\}_{j\in \mathbb{Z}}\in \ell^p(\mathbb{Z})$, $1\leq p \leq \infty$, and let $\{a_j\}_{j\in \mathbb{Z}}$ be an $\ell$-lacunary sequence. Then, for any function $f \in L^\infty(d\mu_k)$ supported in the unit ball $B(0,1)$ and for any $r$ such that $0<2r<1$, we have
  $$\frac{1}{\mu_k(B(0,r))}\int_{B(0,r)} \mathcal{T}^*f(x) \, d\mu_k(x) \leq C\, \Big(\log \frac{2}{r}\Big)^{1/p'}\|f\|_{L^\infty(d\mu_k)},$$
  where $C=C(d, k, G, s, \|\{v_j\}\|_{\ell^p(\mathbb{Z})}, \ell)$.
\end{thm}
We organize the paper as follows. In Section \ref{Sec-prel}, we present the necessary definitions and preliminary results related to the Dunkl setting: orbit distances and orbits of balls, Dunkl operators, the Dunkl transform, Dunkl translations, and Dunkl convolution. Section \ref{sec- Dunkl-Calderon-Zygmund theory} is devoted to presenting known results related to Muckenhoupt weights in the Dunkl setting, weighted Dunkl Calder\'on--Zygmund theory, and Dunkl BMO spaces. In Section \ref{Sec-frac Dunkl heat ker}, we first present known results related to the Dunkl heat kernel and subsequently prove some estimates for the Bochner subordinate function associated with the fractional Dunkl heat semigroup. In this section, we also prove some pointwise estimates for the fractional Dunkl heat semigroup. In the next section, we prove boundedness results for the partial sum operator $\mathcal{T}_{\N}$. In the final section, we first prove a Cotlar-type inequality and then establish the main results, namely Theorem \ref{main thm for maximal op}, Corollary \ref{main thm point wise conv}, and Theorem \ref{main thm local growth}. Here, $A\lesssim B$ means that there exists a constant $C>0$ such that 
$A \leq C B,$ and $A\sim B$ means that $
A\lesssim B \quad \text{and} \quad B\lesssim A.$ Throughout the article, $C$ denotes a universal constant depending only on the parameters involved, which may change from line to line. We always assume that $d\geq 2$.

\section{Preliminaries and notations
}\label{Sec-prel}
Dunkl theory has become standard among specialists. Nevertheless, for the convenience of readers who may not be acquainted with this framework, we briefly present the essential definitions and notation that will be used throughout the paper.

\subsection{The Dunkl setup}

\par Let $\langle \cdot, \cdot \rangle$ denote the standard inner product on $\mathbb{R}^d$, and define the corresponding norm by $|\cdot| = \sqrt{\langle \cdot, \cdot \rangle}$. For any $\lambda \in \mathbb{R}^d$ satisfying $|\lambda|=\sqrt{2}$, consider the mapping $\sigma_\lambda: \mathbb{R}^d \to \mathbb{R}^d$ defined by  
$$\sigma_\lambda (x) = x - \langle x, \lambda \rangle \lambda.$$
This transformation $\sigma_\lambda$ is referred to as the \emph{reflection} corresponding to the vector $\lambda$. In what follows, $R$ denotes a \emph{normalized root system}, meaning that $R \subseteq \mathbb{R}^d \setminus \{0\}$ satisfies the conditions:
\begin{enumerate}[label=(\roman*)]
    \item The set $R$ is finite.
    \item For each $\lambda \in R$, one has $R \cap \mathbb{R}\lambda=\{\lambda, -\lambda\}$.
    \item The relation $\sigma_\lambda(R)=R$ holds for all $\lambda \in R$.
    \item Every $\lambda \in R$ satisfies $|\lambda|^2=2$.
\end{enumerate}

The family of reflections $\{\sigma_\lambda : \lambda \in R\}$ generates a finite subgroup $G \subseteq O(d, \mathbb{R})$, commonly called the \emph{reflection group} associated with $R$.

\par Next, consider a function $k : R\to [0, \infty)$ that is invariant under the action of $G$, i.e., $k(\sigma(\lambda))= k(\lambda)$ for all $\sigma \in G$ and $\lambda \in R$. Such a function is referred to as the \emph{multiplicity function}.

\par Define the $G$-invariant weight function ${\bf v}_k$ by  
${\bf v}_k(x) = \prod _ {\lambda \in R} | \langle x , \lambda \rangle|^{k (\lambda)}$.  
This function is homogeneous of degree $\gamma_k := \sum_{\lambda \in R} k(\lambda)$. Let $d\mu_k(x)$ denote the measure associated with ${\bf v}_k$, defined as $d\mu_k(x)=c_k \, {\bf v}_k(x) \, dx$, where
$$c_k^{-1} = \int_{\mathbb{R}^d} e^{-\frac{|x|^2}{2}} \, {\bf v}_k(x) \, dx.$$
From the homogeneity of both ${\bf v}_k$ and the Lebesgue measure, it follows that $d\mu_k$ is homogeneous of degree $d_k := d + \gamma_k$. A notable complication when dealing with this measure is that the volume of a ball depends not only on its radius but also on its center. More precisely,
\begin{equation*}\label{Volofball}
    \mu_k(B(x,r))\sim \ r^d\prod\limits_{\lambda \in R}\left(|\langle x,\lambda \rangle|+r\right)^{k(\lambda)}.
\end{equation*}
If $0<r_1<r_2$, then the preceding estimates immediately imply that
\begin{eqnarray}\label{VOLRADREL}
   C\left(\frac{r_1}{r_2}\right)^{d_k}\leq  \frac{\mu_k(B(x,r_1))}{\mu_k(B(x,r_2))}\leq C^{-1} \left(\frac{r_1}{r_2}\right)^{d}.
\end{eqnarray}
Also, from the homogeneity of the measure, we have a polar decomposition similar to the classical case, given by
\begin{equation}\label{polar dec}
   \int_{\mathbb{R}^d} f(x)\, d\mu_k(x)= c_k\int_0^\infty \Big(\int_{\mathbb{S}^{d-1}} f(ty)\, {\bf v}_k(y) \, d\sigma(y)\Big)\, t^{d_k-1}\,dt,
\end{equation}
where $d\sigma(y)$ is the usual surface measure on $\mathbb{S}^{d-1}$.
\subsection{Orbit distances and orbits of balls}
The action of the group $G$ allows one to define a modified notion of distance on $\mathbb{R}^d$, namely  
$d_G(x, y) := \min |\sigma(x) - y|$, for any $x, y \in \mathbb{R}^d$.  
Note that $d_G(x, y) = 0$ does not necessarily force $x = y$. Despite this, the function $d_G$ retains the standard metric properties such as nonnegativity, symmetry, and the triangle inequality.
\par For any $r>0$, set
$$
V_G(x,y,r)=\max\left\{\mu_k(B(x,r)),\,\mu_k(B(y,r))\right\}.
$$
As a direct consequence of (\ref{Volofball}), it follows that
\begin{equation}\label{V(x,y,d(x,y) comparison}
V_G(x,y,d_G(x,y))\sim \mu_k(B(x,d_G(x,y))) \sim \mu_k(B(y,d_G(x,y))).
\end{equation}

Let $\mathcal{O}(B)$ denote the orbit of a ball $B$, that is,
$$
\mathcal{O}(B)=\Big\{y\in \mathbb{R}^d : d_G(c_B,y)\leq r(B)\Big\}
=\bigcup_{\sigma \in G}\sigma(B),
$$
where $c_B$ and $r(B)$ denote the centre and radius of $B$, respectively. Since for any $\sigma \in G$, $\mu_k(\sigma(B))=\mu_k(B)$, we have
\begin{equation}\label{orbit volume compa}
\mu_k(B)\leq \mu_k(\mathcal{O}(B))\leq |G|\,\mu_k(B).
\end{equation}

The function $d_G$ is referred to as the \emph{Dunkl metric}, although it does not define a metric on $\mathbb{R}^d$. Clearly, for all $x,y\in \mathbb{R}^{d}$, one has $d_G(x,y)\leq |x-y|$.
\subsection{Dunkl operators, Dunkl kernel and Dunkl transform}

In 1989, Charles F. Dunkl \cite{Dunkl} introduced the family of \emph{Dunkl operators} $\{T_j: 1\leq j\leq d\}$ defined by  
\begin{eqnarray*}
T_j f(x)= \partial_j f(x)+\sum\limits_{\lambda \in R} \frac{k(\lambda )}{2} \lambda _j \frac{f(x) - f(\sigma _\lambda  x)}{\langle \lambda  , x \rangle}, \quad 1\leq j \leq d;
\end{eqnarray*}
where $\lambda =(\lambda_1,\lambda_2,...,\lambda_d)$ and $f\in C^1(\mathbb{R}^d)$.  
These operators extend the classical partial derivatives and reduce to them in the special case when the multiplicity function $k$ vanishes. The \emph{Dunkl gradient} is given by $\nabla_k=(T_1, T_2,\cdots, T_d)$ and the \emph{Dunkl Laplacian} is given by 
$$\Delta_k f(x)= \sum\limits_{j=1}^d T_j^2 f(x)= \Delta f(x)
+ \sum_{\lambda\in R} k(\lambda)
\left(
\frac{\langle\nabla f(x),\lambda\rangle}{\langle\lambda,x\rangle}
-\frac{|\lambda |^2}{2}
\frac{f(x)-f(\sigma_\lambda x)}{\langle\lambda,x\rangle^2}
\right),$$ 
\par For any fixed $y \in \mathbb{R}^d$, it is known (see \cite{DunklIKWRGI}) that there exists a unique real-analytic function $f(x) = E_k(x, y)$ solving the system
\begin{align*}
    \begin{cases}
       T_j f  &=  y_j f, \hspace{.5cm}1\leq j\leq d;\\
       \hfill f(0) &= 1.
    \end{cases}
\end{align*}
The function $E_k$ is called the \emph{Dunkl kernel}, and it plays a role analogous to the exponential function in this setting. In particular, it satisfies several properties similar to those of the exponential, namely:
\begin{enumerate}[label=(\roman*)]
\item $E_k$ admits a unique holomorphic extension to $\mathbb{C}^d \times \mathbb{C}^d$.
\item $E_k(x,y) = E_k(y,x)$ for any $x, y \in \mathbb{C}^d$.
\item $E_k(x,0) = 1$ for any $x \in \mathbb{C}^d$.
\item $E_k(tx,y) = E_k(x,ty)$ for any $x, y \in \mathbb{C}^d$ and for any $t \in \mathbb{C}$. 
\item $|\partial^\alpha E_k(ix,\cdot)(y)| \leq |x|^{|\alpha|},$ for any $x, y\in \mathbb{R}^d$ and any multi-index $\alpha \in \mathbb{N}_0^d$.
\end{enumerate}
These statements are classical, and detailed proofs are available in \cite{DunklIKWRGI, RoslerDOTA}.

\par For $1 \leq p < \infty$, we write $L^p(d\mu_k)$ for the space of $p$-integrable functions on $\mathbb{R}^d$ with respect to the measure $d\mu_k$, while $L^\infty(d\mu_k)$ denotes the corresponding space of essentially bounded functions on $\mathbb{R}^d$.
In analogy with the Fourier transform, the \emph{Dunkl transform} associated with a function $f$ is defined by
\begin{eqnarray*}
\mathcal{F}_kf(\xi)=c_k\,\int_{\mathbb{R}^d}f(x)E_k(-i\xi,x)\,{\bf v}_k(x)\,dx =\int_{\mathbb{R}^d}f(x)E_k(-i\xi,x)\,d\mu_k(x).
\end{eqnarray*}
Owing to the properties of the Dunkl kernel, this expression is well-defined for every $f \in L^1(d\mu_k)$. In the special situation where $k \equiv 0$, the operator $\mathcal{F}_k$ coincides with the classical Fourier transform, which shows that the Dunkl transform extends the usual Fourier framework. We record below several fundamental properties of $\mathcal{F}_k$ that parallel those of the Fourier transform (see \cite{deJeuTDT, RoslerDOTA, XuBook} for further details).
\begin{enumerate}[label=(\roman*)]
    \item The inversion identity
    $$f(x)=\mathcal{F}_k^{-1}(\mathcal{F}_kf)(x)=\int_{\mathbb{R}^d}\mathcal{F}_kf(\xi)E_k(i\xi,x)\,d\mu_k(\xi),$$
    is valid whenever $f, \, \mathcal{F}_kf \in L^1(d\mu_k)$.
    
    \item The transform $\mathcal{F}_k$ preserves the $L^2$-norm, i.e.,
    $||\mathcal{F}_kf||_{L^2(d\mu_k)}=||f||_{L^2(d\mu_k)}$.
    
    \item The scaling relation $\mathcal{F}_k (f_t)(\xi)= \mathcal{F}_k\,f (t\xi)$ holds, where $f_t(x) := t^{-d_k} f(x/t)$, $t>0$.
    
    \item Radial symmetry is preserved: if $f$ is radial, then so is $\mathcal{F}_k f$.
\end{enumerate}

\subsection{Dunkl translation and Dunkl convolution}

Since the measure $d\mu_k$ fails to be invariant under the standard translation on $\mathbb{R}^d$, an alternative notion of translation is required in this setting. This leads to the introduction of the \emph{Dunkl translation} operator, which plays a central role in harmonic analysis in the Dunkl setting. More precisely, the operator $\tau^k_x: L^2(d\mu_k) \to L^2(d\mu_k)$ is defined through the relation
$$\mathcal{F}_k(\tau^k_x f)(y) = E_k(ix,y) \mathcal{F}_k f(y).$$
We recall below several basic properties of $\tau^k_x$ that will be useful later; see \cite{ThangaveluCOMF} for proofs and further discussion.

\begin{enumerate}[label=(\roman*)]
\item $\tau^k_0=I$, that is, $\tau^k_0$ recovers the identity operator.

\item \label{pointwise formula for Dunkl trans} 
For functions $f\in L^2(d\mu_k)$ satisfying $f, \, \mathcal{F}_kf \in L^1(d\mu_k)$, the pointwise representation
$$
\tau^k_yf(x)=\tau^k_x f(y) = \int_{\mathbb{R}^d} E_k(ix, \xi) E_k(iy, \xi) \mathcal{F}_k f(\xi) \, d\mu_k(\xi)
$$
holds.

\item \label{positivity of Dunkl trans} 
If $f$ is a suitable radial function with $f \geq 0$ a.e., then one also has $\tau^k_x f \geq 0$ a.e.

\item \label{bddness of Dunkl trans on radial Lp}
For radial functions, the operator $\tau^k_x$ acts boundedly on $L^p(d\mu_k)$ for every $1 \leq p \leq \infty$ (see \cite{GorbachevPLBDTGTOAIA, ThangaveluCOMF}). In contrast, boundedness for general (nonradial) $L^p$-functions remains unresolved.
\end{enumerate}
Moreover, R\"osler established the following explicit formula for the Dunkl translation of radial Schwartz functions of the form $f(x)=\widetilde{f}(|x|)$ (see \cite{RoslerAPRF}):
\begin{equation*}
\tau^k_x f(-y)=\int\limits_{\mathbb{R}^d}(\widetilde{f}\circ \mathcal{A})(x,y,\eta)\,d\mu_x(\eta),
\end{equation*}
where $\mathcal{A}(x,y,\eta)=\sqrt{|x|^2+|y|^2-2\langle y, \eta \rangle}$, and $\mu_x$ is a probability measure supported on the convex hull of the set $\{ \sigma(x) : \sigma \in G \}$. This representation implies that $\tau^k_x f \geq 0$ whenever $f$ is radial and nonnegative.

\par Let $f,g \in L^2(d\mu_k)$. Their \emph{Dunkl convolution} is defined as
$$f\ast_kg(x)=\int_{\mathbb{R}^d}f(y)\tau^k_xg(-y)\,d\mu_k(y).$$
This convolution operation satisfies several basic properties, which we record below:
\begin{enumerate}[label=(\roman*)]
\item The convolution is symmetric, that is,
$f\ast_kg(x)=g\ast_kf(x)$ for all $f, g \in L^2(d\mu_k)$.

\item\label{dunkl conv is commutative} The Dunkl transform converts convolution into pointwise multiplication, that is, 
$\mathcal{F}_k(f\ast_kg)(\xi)=\mathcal{F}_kf(\xi) \mathcal{F}_kg(\xi)$ for all $f, g \in L^2(d\mu_k)$.
\end{enumerate}

\section{Weighted Dunkl--Calder\'on--Zygmund theory and Dunkl BMO spaces}\label{sec- Dunkl-Calderon-Zygmund theory}
A Muckenhoupt weight is a positive function that controls how averages behave under weighted integration. These weights arise naturally in the study of singular integrals, maximal operators, and partial differential equations. We begin this section with an important class of weight functions in the Dunkl setting.
\begin{defn}
Let $1\leq p<\infty$ and $w$ be a nonnegative locally integrable function on $\mathbb{R}^d$. Then $w$ is said to belong to the class $A^k_p$, if it satisfies
$$\sup_{B\subset \mathbb{R}^d }\Big(  \frac{1}{ \mu_k(B)} \int_{B} w(y) d\mu_k(y) \Big)\Big(  \frac{1}{ \mu_k(B)} \int\limits_{B}w(y)^{1-p'}d\mu_k(y) \Big)^{p-1} < \infty,$$
when $p=1$, $\Big(  \frac{1}{ \mu_k(B)} \int_{B}w(y)^{1-p'}d\mu_k(y) \Big)^{p-1}$ is understood as $\left(\inf\limits_{B}w\right)^{-1}.$
\end{defn}

\subsection{Weighted Dunkl--Calder\'on--Zygmund theory}
Weighted Calder\'on--Zygmund theory studies how singular integral operators act on weighted Lebesgue spaces. We state here the weighted Calder\'on--Zygmund theory in the Dunkl setting from \cite{SumanWBMTIDSVSI}.
\begin{defn}\label{defn of bilinear Dunkl cal zyg ope}
 A Dunkl--Calder\'on--Zygmund operator is a function $\mathcal{T} :\mathcal{S}(\mathbb{R}^d)\to\mathcal{S}'(\mathbb{R}^d)$ such that for $f\in C_c^{\infty}(\mathbb{R}^d)$ with $\sigma (x)\notin  {\rm supp}\,f $ for all  $\sigma\in G$, $\mathcal{T}$ can be represented as 
  $$\mathcal{T}f(x)=\int_{\mathbb{R}^{d}}K(x,y) f(y)\,d\mu_k(y),$$
  where $K$ is a function defined away from the set $\mathcal{O}(\bigtriangleup _{2})$
  $$:=\big\{(x,y)\in \mathbb{R}^{2d}:x= \sigma (y)\text{, for some }\sigma  \in G\big\}$$
  and for any $x, x', y, y'\in \mathbb{R}^d$, $K$ satisfies the following regularity conditions for some $0<\varepsilon \leq 1$:
 \begin{eqnarray*}
|K(x,y)|
&\leq & \frac{C}{\mu_k\big(B(x, d_G(x,y))\big)}\left[\frac{d_G(x,y)}{|x-y|}\right]^\varepsilon
\end{eqnarray*}
 for $d_G(x,y)>0;$
\begin{eqnarray*}
 |K(x,y)-K(x,y')|
&\leq & \frac{C}{\mu_k\big(B(x, d_G(x,y))\big)}\left[\frac{|y-y'|}{|x-y|}\right]^\varepsilon
\end{eqnarray*}
for $|y-y'|<d_G(x,y)/2$; and 
\begin{eqnarray*}
 |K(x,y)-K(x',y)|
&\leq & \frac{C}{\mu_k\big(B(x, d_G(x,y))\big)}\left[\frac{|x-x'|}{|x-y|}\right]^\varepsilon
\end{eqnarray*}
for $|x-x'|<d_G(x,y)/2$.
\end{defn}

The boundedness of $m$-linear Dunkl–Calder\'on-–Zygmund type operators with respect to $m$-fold $G$-invariant Muckenhoupt weights was established in \cite{SumanWBMTIDSVSI} (see also \cite{TanCOTLBIDS} for the linear unweighted case). In the present article, however, we restrict our attention to the linear case. Since these results will be used later, we state the corresponding boundedness theorem below for completeness.
\begin{thm}\label{weighted calderon Zygmund bddness thm}\cite[Theorem 5.3]{SumanWBMTIDSVSI}
    Let $\mathcal{T}$ map from $L^{q}(\mathbb{R}^d,\,d\mu_k)$ to $L^{q,\,\infty}(\mathbb{R}^d,\, d\mu_k)$ for some $q,$ satisfying $1\leq q<\infty$. Then we have the following results.
    \begin{enumerate}[label=(\roman*)]
    \item \label{CZ bound p} If $1<p<\infty$ and $w$ is a $G$-invariant weight such that $w\in A^k_p$, then for any $f\in L^p(\mathbb{R}^d, w\,d\mu_k)$ we have the strong type inequality
    $$\Big(\int_{\mathbb{R}^d}|\mathcal{T}f(x)|^p w(x)\, d\mu_k(x)  \Big)^{1/p}\leq C\, \Big(\int_{\mathbb{R}^d}|f(x)|^p w(x)\, d\mu_k(x)  \Big)^{1/p}.$$

    \item \label{CZ bound 1} If  $w$ is a $G$-invariant weight such that $w\in A^k_1$, then for any $f\in L^1(\mathbb{R}^d, w\,d\mu_k)$ we have the weak type inequality
    $$\sup\limits_{t>0}\,t\Big(\int\limits_{\{|\mathcal{T}f(x)|>t\}} w(x)\, d\mu_k(x)  \Big)\leq C\, \int_{\mathbb{R}^d}|f(x)| w(x)\, d\mu_k(x).$$ 
    \end{enumerate}
Here, $C$ depends on the constant arising from the initial $L^{q}(\mathbb{R}^d,\,d\mu_k)$ to $L^{q,\infty}(\mathbb{R}^d,\,d\mu_k)$ boundedness, and on the constants appearing in the regularity condition of $K$.
\end{thm}
\subsection{Dunkl BMO spaces}\label{subsec-BMO sp}
BMO spaces also arise naturally in the study of singular integral operators. In the Dunkl setting, there are two different types of BMO spaces, which we describe below.
\begin{equation*}
\mathrm{BMO}_k(\mathbb{R}^{d})
:=
\big\{
f \in L_{\text{loc}}^{1}(\mathbb{R}^{d}, d\mu_k) :
\|f\|_{\mathrm{BMO}_k} 
:=
\sup_{B\, \subseteq \, \mathbb{R}^{d}}
\frac{1}{\mu_k(B)}
\int_{B}
|f(x)-f_B|\, d\mu_k(x)
< \infty
\big\},
\end{equation*}
where
$$f_{B}
:=
\frac{1}{\mu_k(B)}
\int_{B} f(x)\, d\mu_k(x).$$

Also, using the Dunkl metric $d_G(x,y)$, we define the following BMO space:
\begin{equation*}
\mathrm{BMO}_{G,\,k}(\mathbb{R}^{d})
:=
\big\{
f \in L_{\text{loc}}^{1}(\mathbb{R}^{d}, d\mu_k) :
\|f\|_{\mathrm{BMO}_{G,\,k}}
< \infty
\big\},
\end{equation*}
where
$$\|f\|_{\mathrm{BMO}_{G,\,k}} 
:=
\sup_{B\, \subseteq \, \mathbb{R}^{d}}
\frac{1}{\mu_k(\mathcal{O}(B))}
\int_{\mathcal{O}(B)}
|f(x)-f_{\mathcal{O}(B)}|\, d\mu_k(x)$$ and 
$$f_{\mathcal{O}(B)}
:=
\frac{1}{\mu_k(\mathcal{O}(B))}
\int_{\mathcal{O}(B)} f(x)\, d\mu_k(x).$$
As usual, $f \in \mathrm{BMO}_{k}(\mathbb{R}^{d})$ denotes an element of the quotient space $\mathrm{BMO}_{k}(\mathbb{R}^{d})$ modulo constants. Similarly, $f \in \mathrm{BMO}_{G,\,k}(\mathbb{R}^{d})$ denotes an element of the quotient space $\mathrm{BMO}_{G,\,k}(\mathbb{R}^{d})$ modulo constants. We also have the equivalence of the norms as follows.
$$\|f\|_{\mathrm{BMO}_k} 
\sim
\sup_{B\, \subseteq \, \mathbb{R}^{d}} \inf_{c \in \mathbb{C}}
\frac{1}{\mu_k(B)}
\int_{B}
|f(x)-c\,|\, d\mu_k(x)$$
and 
$$\|f\|_{\mathrm{BMO}_{G,\,k}} 
\sim
\sup_{B\, \subseteq \, \mathbb{R}^{d}} \inf_{c \in \mathbb{C}}
\frac{1}{\mu_k(\mathcal{O}(B))}
\int_{\mathcal{O}(B)}
|f(x)-c\,|\, d\mu_k(x)$$
It is not too hard to see that
$$
L^{\infty}(\mathbb{R}^{d}, d\mu_k)
\subset
\mathrm{BMO}_{G,\,k}(\mathbb{R}^{d})
\subset
\mathrm{BMO}_{k}(\mathbb{R}^{d}),
$$
that is 
$$\|\cdot\|_{\mathrm{BMO}_{k}} \lesssim \|\cdot\|_{\mathrm{BMO}_{G,\,k}}  \lesssim \|\cdot\|_{L^\infty(d\mu_k)},$$
and the inclusions are strict. Moreover, if $f \in \mathrm{BMO}_{k}(\mathbb{R}^{d})$ is $G$-invariant, then
$f \in \mathrm{BMO}_{G,\,k}$ $(\mathbb{R}^{d})$.
For details, we refer to \cite[Proposition 7.4]{JiuTDOTHSAWTDO}.


\section{Estimates for the fractional Dunkl heat kernel}\label{Sec-frac Dunkl heat ker} 
\subsection{Dunkl heat kernel}  

We begin by recalling several facts concerning the Dunkl heat kernel and the associated heat semigroup, following the work of R\"osler \cite{RoslerGHPATHEFDO}. The \emph{Dunkl heat semigroup} $\{e^{t\,\Delta_k}\}_{t\geq 0}$ is expressed as
$$e^{t\,\Delta_k}f(x)= \mathcal{F}_k^{-1}(e^{-t|\cdot|^2} \mathcal{F}_kf)(x).$$
As discussed earlier in Section \ref{Sec-intro}, the function $u(x,t)=e^{t\,\Delta_k} f(x)$ solves the Dunkl heat equation
\begin{align*}\label{Dunkl heat eq}
\begin{cases}
\partial_t u(x,t) -\Delta_k \,u(x,t) &= 0,\, \quad (x,t)\in \mathbb{R}^d \times (0,\infty),\\
\hfill u(x,0) &= f(x).
\end{cases}
\end{align*}

Applying the Dunkl transform shows that this semigroup admits a convolution representation:
$$e^{t\,\Delta_k}f(x)= f \ast_k h_t(x)= \int_{\mathbb{R}^d}\tau^k_xh_t(-y)f(y) \, d\mu_k(y),$$
where $\tau^k_xh_t(-y)$ is called the \emph{Dunkl heat kernel}. The function $h_t$ itself is given by
$$h_t(x)=\mathcal{F}_k^{-1}(e^{-t|\cdot|^2})(x)= (2t)^{-d_k/2}e^{-|x|^2/(4t)}.$$ 
In particular, $h_t$ is a nonnegative radial Schwartz function. Clearly, from the properties of Dunkl translation, $\tau^k_xh_t(-y)\geq 0$. It also satisfies the normalization condition
\begin{equation}\label{int of heat ker is 1}
\int_{\mathbb{R}^d}h_t(x)\, d\mu_k(x)=1 .
\end{equation}

We will make use of the following estimates for the Dunkl heat kernel.

\begin{prop}\label{Dunkl heat ker esti} \cite[Theorem 3.1]{HejnaROADFTHSHITRDS}
Let $t>0$ and $x, y\in \mathbb{R}^d$. Then
\begin{enumerate}[label=(\roman*)]
\item \label{heat i} \leavevmode\vspace*{-\dimexpr\abovedisplayskip + \baselineskip}
$$\tau^k_x h_t(-y) \lesssim \frac{1}{\mu_k(B(x, \sqrt{t}))}\,\Big(1+\frac{|x-y|}{\sqrt{t}}\Big)^{-2} \exp{(-c\,{d_G(x,\, y)^2}/{t})}.$$

\item \label{heat ii} 
$$|\tau^k_x h_t(-y)-\tau^k_x h_t(-y')| \lesssim \frac{|y-y'|}{\sqrt{t}}\frac{1}{\mu_k(B(x, \sqrt{t}))}\, \Big(1+\frac{|x-y|}{\sqrt{t}}\Big)^{-2}\exp{(-c\,{d_G(x,\, y)^2}/{t})},$$
whenever $|y-y'|\leq \sqrt{t}$.
\end{enumerate}
\end{prop}

\subsection{ Fractional Dunkl heat kernel}\label{subsec-fractional Dunkl heat}

We now turn to semigroups generated by fractional powers of $(-\Delta_k)$. For $0<s\leq 1$, define
$$e^{-t\,(-\Delta_k)^{s}}f(x):= \mathcal{F}_k^{-1}(e^{-t|\cdot|^{2s}} \mathcal{F}_kf)(x).$$
In fact, $e^{-t\,(-\Delta_k)^{s}}f(x)$ solves the fractional Dunkl heat equation
\begin{align*}
    \begin{cases}
       \partial_t u(x,t) + (-\Delta_k)^{s} \,u(x,t) &= 0,\, \quad (x,t)\in \mathbb{R}^d \times (0,\infty),\\
       \hfill u(x,0) &= f(x).
    \end{cases}
    \end{align*}
When $s=1$, this construction reduces to the Dunkl heat semigroup $\{e^{t\,\Delta_k}\}_{t\geq 0}$.

Since the Dunkl heat semigroup is known to be a strongly continuous contraction semigroup on $L^p(d\mu_k), 1\leq p< \infty$ (see \cite{RoslerDOTA}), Bochner’s subordination formula \cite{BochnerDEASP} yields the representation
\begin{equation}\label{Bochner sub ord formula}
e^{-t\,(-\Delta_k)^{s}}f(x)= \int_0^\infty e^{\,u\,\Delta_k}f(x)\, \eta_{t,s}(u) \, du, \quad 0<s<1.
\end{equation}
Here, $\eta_{t,s}$ is a probability density on $(0,\infty)$. Proceeding as before, one obtains a convolution formula
$$e^{-t\,(-\Delta_k)^{s}}f(x)= f \ast_k h_{t,s}(x)= \int_{\mathbb{R}^d}\tau^k_xh_{t,s}(-y)f(y) \, d\mu_k(y),$$
where $\tau^k_x h_{t, s}(-y)$ is referred to as the \emph{fractional Dunkl heat kernel}, and
$$h_{t, s}(x) = \mathcal{F}_k^{-1}(e^{-t|\cdot|^{2s}})(x) .$$

Although no explicit formula for $h_{t,s}$ is available, it admits the integral representation
$$h_{t, s}(x)= \int_0^\infty h_u(x)\, \eta_{t,s}(u) \, du .$$
So, $h_{t, s}$ is also a nonnegative radial function. In particular, when $s=1/2$, the subordinate function is explicitly known, and in this case the kernel is known as the \emph{Dunkl Poisson kernel} (see \cite[Section 5]{HejnaHFCHF} for details).
\par In the next proposition, we collect a few important properties of the subordinate function that will be needed later.
\begin{prop}\label{Subordi func esti}
The function $\eta_{t,s}$ satisfies the following properties.
\begin{enumerate}[label=(\roman*)]
\item \label{eta 1} For all $t, u>0$,
$$\eta_{t,s}(u) = t^{-1/s}\,\eta_{1,s}(u/t^{1/s}).$$

\item \label{eta 4} For all $t, u>0$, 
$$\eta_{t,s}(u) \lesssim {t}/{u^{1+ s}}.$$

\item \label{eta 5} There exists $c_s>0$ such that for $u \leq t^{1/s}$,
$$\eta_{t,s}(u) \sim t^{\frac{1}{2(1-s)}}u^{-\frac{2-s}{2(1-s)}}\, \exp({-c_s\,t^{\frac{1}{1-s}}u^{-\frac{s}{1-s}} }).$$
\end{enumerate}
\end{prop}
The proofs of \ref{eta 1} and \ref{eta 4} can be found in \cite[p. 20]{GrigoryanHKAFTOMMS}; see also \cite[p. 4]{RejebSRRTTFDL}. A proof of \ref{eta 5} appears in \cite[p. 89]{GraczykTDEFSPOSS}.

\subsection{ Pointwise estimates for the fractional Dunkl heat kernel} 
In this section, we prove some pointwise bounds for the fractional Dunkl heat kernel and its time derivatives. In the special case where $s ={1}/{2}$, these recover the results for the Dunkl Poisson kernel in \cite[Proposition 5.1]{HejnaHFCHF}.
\begin{prop}\label{esti for fractional Dunkl heat kernel}
 For any $x, y\in \mathbb{R}^d$,
 \begin{enumerate}[label=(\roman*)]
 \item \label{frac heat i} the fractional Dunkl heat kernel satisfies 
 
 \begin{equation*}
    \tau^k_xh_{t,s}(-y) \lesssim \frac{1}{\mu_k \big(B(x, t^{\frac{1}{2s}}+d_G(x,y))\big)} \frac{t}{\big(t^{\frac{1}{2s}}+d_G(x,y)\big)^{2s}}.
 \end{equation*}
 \item \label{frac heat ii} the time derivative of the fractional Dunkl heat kernel satisfies 
 \begin{equation*}
   \big|\frac{\partial}{\partial t}\tau^k_xh_{t,s}(-y)\big| \lesssim \frac{1}{\mu_k \big(B(x, t^{\frac{1}{2s}}+d_G(x,y))\big)} \frac{1}{\big(t^{\frac{1}{2s}}+d_G(x,y)\big)^{2s}}.  
 \end{equation*}
 \end{enumerate}
\end{prop}
\begin{proof}[Proof of \ref{frac heat i}] {\bf Case I:} $t^{\frac{1}{2s}}\geq d_G(x,y)$.
    
\par   In this case $t^{\frac{1}{2s}}+d_G(x,y)\sim t^{\frac{1}{2s}}$.
Using the subordination formula, we break the integral in to two parts
\begin{eqnarray*}
  \tau^k_xh_{t, s}(-y)&=&\int_0^{t^{1/s}} \tau^k_xh_u(-y)\, \eta_{t,s}(u) \, du+ \int_{t^{1/s}}^\infty  \tau^k_xh_u(-y)\, \eta_{t,s}(u) \, du\\
  &=:& I + II.
\end{eqnarray*}

Using Proposition \ref{Dunkl heat ker esti} \ref{heat i}, \eqref{VOLRADREL},  and Proposition \ref{Subordi func esti} \ref{eta 5}, we dominate the first term by a simplified integral as follows
\begin{eqnarray*}
 I &\lesssim &  \int_0^{t^{1/s}} \frac{1}{\mu_k(B(x, \sqrt{u}))}\, \exp{(-c\,{d_G(x,\, y)^2}/{u})} \, \eta_{t,s}(u) \, du\\
 &\lesssim & \frac{1}{\mu_k(B(x, t^{\frac{1}{2s}}))}\, \int_0^{t^{1/s}} \Big( \frac{t^{\frac{1}{2s}}}{\sqrt{u}}\Big)^{d_k} \eta_{t,s}(u) \, du\\
 &\lesssim & \frac{1}{\mu_k(B(x, t^{\frac{1}{2s}}))}\, \int_0^{t^{1/s}} \Big( \frac{t^{\frac{1}{2s}}}{\sqrt{u}}\Big)^{d_k} t^{\frac{1}{2(1-s)}}u^{-\frac{2-s}{2(1-s)}}\, \exp({-c_s\,t^{\frac{1}{1-s}}u^{-\frac{s}{1-s}} }) \, du\\
 &= & \frac{t}{\mu_k(B(x, t^{\frac{1}{2s}}))}\,t^{A} \int_0^{t^{1/s}} u^{-B}\, \exp({-c_s\,t^{\frac{1}{1-s}}u^{-\frac{s}{1-s}} }) \, du,
 \end{eqnarray*}
where 
$$A=\frac{d_k}{2s}+ \frac{2s-1}{2-2s},$$
$$B=\frac{d_k}{2}+ \frac{2-s}{2(1-s)}>1.$$

Now making a change of variable $t^{\frac{1}{1-s}}u^{-\frac{s}{1-s}}$ going to $u$, we obtain 
\begin{eqnarray*}
    I &\lesssim &  \frac{t}{\mu_k(B(x, t^{\frac{1}{2s}}))}\,\frac{1}{t} \int_1 ^\infty u^{\frac{(B-1)(1-s)}{s}-1}e^{-c_s\, u} \,du\\
    &\lesssim &  \frac{t}{\mu_k(B(x, t^{\frac{1}{2s}}))}\,\frac{1}{t}\\
    &\lesssim & \frac{1}{\mu_k \big(B(x, t^{\frac{1}{2s}}+d_G(x,y))\big)} \frac{t}{\big(t^{\frac{1}{2s}}+d_G(x,y)\big)^{2s}}.
\end{eqnarray*}
For the second term, using Proposition \ref{Dunkl heat ker esti} \ref{heat i}, \eqref{VOLRADREL},  and Proposition \ref{Subordi func esti} \ref{eta 4}, we directly get
\begin{eqnarray*}
    II &\lesssim &  \int_{t^{1/s}}^\infty \frac{1}{\mu_k(B(x, \sqrt{u}))}\, \exp{(-c\,{d_G(x,\, y)^2}/{u})} \, \eta_{t,s}(u) \, du\\
 &\lesssim & \frac{1}{\mu_k(B(x, t^{\frac{1}{2s}}))}\, \int_{t^{1/s}}^\infty \Big( \frac{t^{\frac{1}{2s}}}{\sqrt{u}}\Big)^{d} \eta_{t,s}(u) \, du\\
 &\lesssim & \frac{1}{\mu_k(B(x, t^{\frac{1}{2s}}))}\, \int_{t^{1/s}}^\infty \Big( \frac{t^{\frac{1}{2s}}}{\sqrt{u}}\Big)^{d}\frac{t}{u^{1+s}}\, du\\
 &= &  \frac{1}{\mu_k(B(x, t^{\frac{1}{2s}}))}\\
    &\lesssim & \frac{1}{\mu_k \big(B(x, t^{\frac{1}{2s}}+d_G(x,y))\big)} \frac{t}{\big(t^{\frac{1}{2s}}+d_G(x,y)\big)^{2s}}.
\end{eqnarray*}
 {\bf Case II:} $t^{\frac{1}{2s}}< d_G(x,y)$.
 \par   In this case $t^{\frac{1}{2s}}+d_G(x,y)\sim d_G(x,y)$. Here also, we break the integral in to two parts
\begin{eqnarray*}
  \tau^k_xh_{t, s}(-y)&=&\int_0^{d_G(x,y)^2} \tau^k_xh_u(-y)\, \eta_{t,s}(u) \, du+ \int_{d_G(x,y)^2}^\infty  \tau^k_xh_u(-y)\, \eta_{t,s}(u) \, du\\
  &=:& I' + II'.
\end{eqnarray*}
Similarly, using Proposition \ref{Dunkl heat ker esti} \ref{heat i}, \eqref{VOLRADREL}, properties of the exponential function  and Proposition \ref{Subordi func esti} \ref{eta 4}, we compute 
\begin{eqnarray*}
    I' &\lesssim &  \int_0^{d_G(x,y)^2} \frac{1}{\mu_k(B(x, \sqrt{u}))}\, \exp{(-c\,{d_G(x,\, y)^2}/{u})} \, \eta_{t,s}(u) \, du\\
 &\lesssim & \frac{1}{\mu_k(B(x, d_G(x,y)))}\, \int_0^{d_G(x,y)^2} \Big( \frac{d_G(x,y)}{\sqrt{u}}\Big)^{d_k} \exp{(-c\,{d_G(x,\, y)^2}/{u})}\,\eta_{t,s}(u) \, du\\
 &\lesssim & \frac{1}{\mu_k(B(x, d_G(x,y)))}\, \int_0^{d_G(x,y)^2} \Big( \frac{d_G(x,y)}{\sqrt{u}}\Big)^{d_k}\Big(\frac{u}{d_G(x,y)^2}\Big)^{1+d_k/2}\frac{t}{u^{1+s}}\, du\\
 &= &  \frac{t}{\mu_k(B(x, d_G(x,y)))} \frac{1}{d_G(x,y)^2}\int_0^{d_G(x,y)^2} u^{-s}\, du\\
    &\lesssim & \frac{1}{\mu_k \big(B(x, t^{\frac{1}{2s}}+d_G(x,y))\big)} \frac{t}{\big(t^{\frac{1}{2s}}+d_G(x,y)\big)^{2s}}.
\end{eqnarray*}
The estimate for $II'$ can be done in the same way as $II$, and hence omitted. 
\end{proof}

\begin{proof}[Proof of \ref{frac heat ii}] 

We first establish an estimate for the time derivative of the
subordination density. We claim that there exist $b\in(0,1)$ and
$C>0$, depending only on $s$, such that for any $t, u>0$,
\begin{equation}\label{eq:subordination-time-derivative}
\left|\partial_t\eta_{t,s}(u)\right|
\leq \frac{C}{t}\eta_{\,bt,s}(u).
\end{equation}
To prove this, let $\beta={s}/{(1-s)}$. Recall the Zolotarev representation of the positive $s$-stable
density (see, for instance,
\cite[Section~4.4]{UchaikinBook}):
\begin{equation}\label{eq:zolo-density}
\eta_{1,s}(u)
=
\beta u^{-\beta-1}
\int_0^1
\varsigma_s(\theta)
\exp\left(-\varsigma_s(\theta)u^{-\beta}\right)
\,d\theta,
\end{equation}
where
$$\varsigma_s(\theta)
=
\left(
\frac{
[\sin(\pi s\theta)]^s
[\sin(\pi(1-s)\theta)]^{1-s}}
{\sin(\pi\theta)}
\right)^{\frac{1}{1-s}},
\qquad 0<\theta<1.$$
In particular, $\varsigma_s(\theta)>0$. Differentiation in \eqref{eq:zolo-density} is justified locally
uniformly for $u>0$. Indeed, on every compact subinterval of
$(0,\infty)$, the differentiated integrand is dominated by a
constant multiple of $\varsigma_s(\theta)
\bigl(1+\varsigma_s(\theta)\bigr)
\exp{(-c\,\varsigma_s(\theta))}$,
which is integrable on $(0,1)$. Define
$$
g_s(u):=\eta_{1,s}(u)+u\eta_{1,s}'(u).
$$
Differentiating \eqref{eq:zolo-density}, we obtain
\begin{equation}\label{eq:gs-zolo}
g_s(u)
=
\beta^2u^{-\beta-1}
\int_0^1
\varsigma_s(\theta)
\left(\varsigma_s(\theta)u^{-\beta}-1\right)
e^{-\varsigma_s(\theta)u^{-\beta}}
\,d\theta.
\end{equation}
Fix $a>1$. Since $a^{-\beta}<1$, for every $z\geq0$ we have
$$
(1+z)e^{-z}
\leq C_{a,s}e^{-a^{-\beta}z}.
$$
Consequently, by \eqref{eq:gs-zolo},
\begin{align*}
|g_s(u)|
&\leq
\beta^2u^{-\beta-1}
\int_0^1
\varsigma_s(\theta)
\left(1+\varsigma_s(\theta)u^{-\beta}\right)
e^{-\varsigma_s(\theta)u^{-\beta}}
\,d\theta\\
&\lesssim
u^{-\beta-1}
\int_0^1
\varsigma_s(\theta)
e^{-a^{-\beta}\varsigma_s(\theta)u^{-\beta}}
\,d\theta\\
&=
u^{-\beta-1}
\int_0^1
\varsigma_s(\theta)
e^{-\varsigma_s(\theta)(au)^{-\beta}}
\,d\theta
\lesssim
\eta_{1,s}(au).
\end{align*}
Thus by the scaling property of the subordinate function,
$$\left|\partial_t\eta_{t,s}(u)\right|
\lesssim
\frac{1}{t}\eta_{bt,s}(u),\ \text{where } b=a^{-s}\in(0,1).$$
This proves \eqref{eq:subordination-time-derivative}.

We next justify differentiation under the subordination integral.
Fix $t_0>0$ and let
$t\in[t_0/2,3t_0/2]$. By \eqref{eq:subordination-time-derivative},
$$
|\partial_t\eta_{t,s}(u)|
\lesssim
\frac{1}{t_0}\eta_{bt,s}(u).
$$
Using \eqref{eq:zolo-density} together with the scaling relation,
we may write, for $r>0$,
$$
\eta_{r,s}(u)
=
\beta r^{1/(1-s)}u^{-\beta-1}
\int_0^1
\varsigma_s(\theta)
\exp\left(
-\varsigma_s(\theta)r^{1/(1-s)}u^{-\beta}
\right)
\,d\theta.
$$
Consequently, whenever ${bt_0}/{2}\leq r\leq {3bt_0}/{2}$, we have
$\eta_{r,s}(u)
\lesssim
\eta_{bt_0/2,s}(u)$. Thus
$$
\tau^k_xh_u(-y)|\partial_t\eta_{t,s}(u)|
\lesssim
\frac{1}{t_0}
\tau^k_xh_u(-y)\eta_{bt_0/2,s}(u).
$$
The function on the right-hand side is integrable in $u$, since,
by the subordination formula,
$$
\int_0^\infty
\tau^k_xh_u(-y)\eta_{bt_0/2,s}(u)\,du
=
\tau_x^k h_{bt_0/2,s}(-y)
<\infty.
$$
Hence dominated convergence yields
$$
\partial_t\tau_x^k h_{t,s}(-y)
=
\int_0^\infty
\tau^k_xh_u(-y)\partial_t\eta_{t,s}(u)\,du.
$$
Using the positivity of the Dunkl heat kernel and
\eqref{eq:subordination-time-derivative}, we obtain
\begin{align*}
\left|
\partial_t\tau_x^k h_{t,s}(-y)
\right|
&\leq
\int_0^\infty
\tau^k_xh_u(-y)|\partial_t\eta_{t,s}(u)|\,du
\\
&\lesssim
\frac{1}{t}
\int_0^\infty
\tau^k_xh_u(-y)\eta_{\,bt,s}(u)\,du
\\
&=
\frac{1}{t}\tau_x^k h_{bt,s}(-y).
\label{eq:kernel-time-by-kernel}
\end{align*}
Since $b<1$, the bounds in part \ref{frac heat ii} follow from those in part \ref{frac heat i}, together with the doubling property.
\end{proof}
\par We conclude this subsection by deriving a new estimate involving the integral of the time derivative of the subordinate function $\eta_{t,s}$.
\begin{prop}\label{esti of deriv of subordinate func} For any $u>0$, we have the estimate
   $$\int_0^\infty \big|\frac{\partial}{\partial t}\eta_{t,s}(u)\big|\, dt \lesssim \frac{1}{u}.$$
\end{prop}
\begin{proof}
By the estimate for the subordination density established in
the proof of Proposition \ref{esti for fractional Dunkl heat kernel} \ref{frac heat ii}, there exist $b\in(0,1)$ and
$C_s>0$, depending only on $s$, such that for any $t,u>0$,
$$|\partial_t\eta_{t,s}(u)|
    \le \frac{C_s}{t}\eta_{bt,s}(u).$$
Consequently, using the scaling property of $\eta_{t,s}$, we obtain
$$
\begin{aligned}
    \int_0^\infty |\partial_t\eta_{t,s}(u)|\,dt
    &\le C_s\int_0^\infty \eta_{bt,s}(u)\,\frac{dt}{t}\\
    &=C_s\int_0^\infty
      (bt)^{-1/s}\eta_{1,s}\bigl(u(bt)^{-1/s}\bigr)
      \,\frac{dt}{t}.
\end{aligned}
$$
Making the change of variables $v=u(bt)^{-1/s}$, we conclude that
$$\int_0^\infty |\partial_t\eta_{t,s}(u)|\,dt
    \le \frac{C_s s}{u}
       \int_0^\infty \eta_{1,s}(v)\,dv
    =\frac{C_s s}{u}
    \lesssim \frac1u,$$
since $\eta_{1,s}$ is a probability density on $(0,\infty)$.
\end{proof}


\section{Uniform boundedness of the partial sum operators}\label{Sec-boundeness of partial sum op} 
In this section, we study the uniform boundedness of the partial sum operator $\mathcal{T}_{\N}$. We prove results analogous to Theorem \ref{main thm for maximal op}, which will later be used in the proof of Theorem \ref{main thm for maximal op}. Throughout this section (only), we do not assume that the sequence $\{a_j\}_{j\in \mathbb{Z}}$ is lacunary. The condition that $\{a_j\}_{j\in \mathbb{Z}}$ is a lacunary sequence is not required for the results in this section. We begin with the following proposition.

\begin{prop}\label{L2 bddness of TN}
 There exists $C=C(d, k, G, s, \|\{v_j\}\|_{\ell^\infty(\mathbb{Z})})>0$ and independent of $\N$ such that 
 $$\|\mathcal{T}_{\N}f\|_{L^2(d\mu_k)} \leq C \|f\|_{L^2(d\mu_k)}.$$
\end{prop}
\begin{proof}
The proof follows in the same way as in the case $s=1$, by applying the Dunkl--Plancherel formula and Fubini's theorem:
\begin{eqnarray*}
\|\mathcal{T}_{\N}f\|^2_{L^2(d\mu_k)} &=& \int_{\mathbb{R}^d}  \Big|\mathcal{F}_k\Big(\sum\limits_{j=N_1}^{N_2} v_j \Big(e^{-a_{j+1}(-\Delta_k)^s}f - e^{-a_j(-\Delta_k)^s} f\Big)\Big)(\xi)\Big|^2\, d\mu_k(\xi) \\
&=& \int_{\mathbb{R}^d}  \Big|\sum\limits_{j=N_1}^{N_2} v_j \int_{a_j}^{a_{j+1}}\frac{\partial}{\partial t}\mathcal{F}_k\big(e^{-t(-\Delta_k)^s}f \big)(\xi)\, dt\Big|^2\, d\mu_k(\xi) \\
&\leq & \|\{v_j\}\|^2_{\ell^\infty(\mathbb{Z})}\int_{\mathbb{R}^d}  \Big( \int_{0}^{\infty}|\xi|^{2s}\, e^{-t\,|\xi|^{2s}}\,|\mathcal{F}_kf(\xi)|\, dt\Big)^2\, d\mu_k(\xi)\\
&\lesssim & \|f\|^2_{L^2(d\mu_k)}.
\end{eqnarray*}      
\end{proof}

In the next proposition, we show that the kernel of the operator $\mathcal{T}_{\N}$ satisfies the kernel conditions for Dunkl--Calder\'on--Zygmund operators uniformly in $\N$. We remark that in \cite{DasBODTFDHS}, the authors established these estimates for the kernel of the partial sum operator associated with the Dunkl heat semigroup by first proving that the operator is a Dunkl multiplier operator and then invoking the known results for the kernels of Dunkl multiplier operators. In the fractional heat case, however, we directly derive these estimates for the kernel by using subordination formula. We also remark that such direct and simpler computations are possible in the heat semigroup case as well.
\begin{prop}\label{kernel of TN}
 For any $f\in C_c^\infty (\mathbb{R}^d)$, $\mathcal{T}_{\N}$ can be written as 
 $$\mathcal{T}_{\N}f(x)= \int_{\mathbb{R}^d}K_{\N}(x,y) \, f(y)\, d\mu_k(y),$$
 where $$K_{\N}(x,y)=\sum\limits_{j=N_1}^{N_2} v_j \big(\tau^k_x h_{a_{j+1},s} (-y)- \tau^k_x h_{a_{j},s} (-y)\big).$$
 Also, there exists $C=C(d, k, G, s, \|\{v_j\}\|_{\ell^\infty(\mathbb{Z})})>0$ and independent of $\N$ such that
 \begin{eqnarray}\label{size estimate old}
|K_{\N}(x,y)|
&\leq & \frac{C}{\mu_k\big(B(x, d_G(x,y))\big)}\frac{d_G(x,y)}{|x-y|}
\end{eqnarray}
 for $d_G(x,y)>0;$
\begin{eqnarray}\label{smtness estimate last y changing}
 |K_{\N}(x,y)-K_{\N}(x,y')|
&\leq & \frac{C}{\mu_k\big(B(x, d_G(x,y))\big)}\frac{|y-y'|}{|x-y|}
\end{eqnarray}
for $|y-y'|<d_G(x,y)/2$; and 
\begin{eqnarray}\label{smtness estimate last x changing}
 |K_{\N}(x,y)-K_{\N}(x',y)|
&\leq & \frac{C}{\mu_k\big(B(x, d_G(x,y))\big)}\frac{|x-x'|}{|x-y|}
\end{eqnarray}
for $|x-x'|<d_G(x,y)/2$.
\end{prop}
It suffices to establish \eqref{size estimate old} and \eqref{smtness estimate last y changing}. The estimate \eqref{smtness estimate last x changing} follows similarly from the argument for \eqref{smtness estimate last y changing} by symmetry.
 \begin{proof}[Proof of \eqref{size estimate old}] We note that, by applying Proposition \ref{esti of deriv of subordinate func}, we can write
\begin{eqnarray}\label{int 1 by u factor}
 | K_{\N}(x,y)|&\leq & \sum\limits_{j=N_1}^{N_2} \big|\,v_j \big(\tau^k_x h_{a_{j+1},s} (-y)- \tau^k_x h_{a_{j},s} (-y)\,\big)\big| \\
  &\lesssim & \sum\limits_{j=-\infty}^{\infty} \big|\int_{a_j}^{a_{j+1}}\frac{\partial}{\partial t}\tau^k_xh_{t,s}(-y)\, dt\,\big|\nonumber\\
  &\leq &  \int_0^\infty \tau^k_xh_u(-y)\, \Big(\int_0^\infty\big|\frac{\partial}{\partial t}\eta_{t,s}(u) \big| \,dt \Big)\,du\nonumber\\
  &\lesssim &  \int_0^\infty \tau^k_xh_u(-y)\,\frac{du}{u}.\nonumber
\end{eqnarray}

We apply Proposition \ref{Dunkl heat ker esti} \ref{heat i} in \eqref{int 1 by u factor},  and decompose $K_{\N}(x,y)$ into two parts as follows
\begin{eqnarray*}
   &&| K_{\N}(x,y)|\\
   &\lesssim & \int_0^\infty  \frac{1}{\mu_k(B(x, \sqrt{u}))}\,\Big(1+\frac{|x-y|}{\sqrt{u}}\Big)^{-2} \exp{(-c\,{d_G(x,\, y)^2}/{u})}\,\frac{du}{u}\\
   &\leq & \int_0^\infty  \frac{1}{\mu_k(B(x, \sqrt{u}))}\,\frac{\sqrt{u}}{|x-y|}\, \exp{(-c\,{d_G(x,\, y)^2}/{u})}\,\frac{du}{u}\\
   &\lesssim & \frac{1}{|x-y|}\int_0^\infty  \frac{\exp{(-c\,{d_G(x,\, y)^2}/{u})}}{\mu_k(B(x, \sqrt{u}))}\,\frac{du}{\sqrt{u}}\\
   &= & \frac{1}{|x-y|}\int_0^{d_G(x,y)^2}  \frac{\exp{(-c\,{d_G(x,\, y)^2}/{u})}}{\mu_k(B(x, \sqrt{u}))}\,\frac{du}{\sqrt{u}} \\
   &&+ \frac{1}{|x-y|}\int_{d_G(x,y)^2}^\infty  \frac{\exp{(-c\,{d_G(x,\, y)^2}/{u})}}{\mu_k(B(x, \sqrt{u}))}\,\frac{du}{\sqrt{u}}\\
   &=:& I +II.
\end{eqnarray*}
Applying the inequality \eqref{VOLRADREL} and properties of the exponential function, we obtain the required estimate for the first term as follows
\begin{eqnarray*}
   I &\lesssim & \frac{1}{|x-y|}\, \frac{1}{\mu_k\big(B(x, d_G(x,y))\big)}\int_0^{d_G(x,y)^2} \Big( \frac{d_G(x,y)}{\sqrt{u}}\Big)^{d_k} \exp{(-c\,{d_G(x,\, y)^2}/{u})}\,\frac{du}{\sqrt{u}}\\
   &\lesssim & \frac{1}{|x-y|}\, \frac{1}{\mu_k\big(B(x, d_G(x,y))\big)}\int_0^{d_G(x,y)^2} \Big( \frac{d_G(x,y)}{\sqrt{u}}\Big)^{d_k}\, \Big( \frac{u}{d_G(x,y)^2}\Big)^{\frac{d_k+1}{2}}\,\frac{du}{\sqrt{u}}\\
   &\lesssim & \frac{1}{\mu_k\big(B(x, d_G(x,y))\big)}\frac{d_G(x,y)}{|x-y|}.
\end{eqnarray*}
For the second term, the calculations are even simpler; in fact, by applying \eqref{VOLRADREL}, we obtain
\begin{eqnarray*}
  II &\lesssim & \frac{1}{|x-y|}\, \frac{1}{\mu_k\big(B(x, d_G(x,y))\big)}\int_{d_G(x,y)^2}^\infty \Big( \frac{d_G(x,y)}{\sqrt{u}}\Big)^{d} \exp{(-c\,{d_G(x,\, y)^2}/{u})}\,\frac{du}{\sqrt{u}}\\
   &\lesssim & \frac{1}{|x-y|}\, \frac{1}{\mu_k\big(B(x, d_G(x,y))\big)}\, d_G(x,y)^d\int_{d_G(x,y)^2}^\infty u^{-\frac{d+1}{2}}\, du\\
   &\lesssim & \frac{1}{\mu_k\big(B(x, d_G(x,y))\big)}\frac{d_G(x,y)}{|x-y|}, 
\end{eqnarray*}
 where in the last step we have used the fact that $d\geq 2$. 
\end{proof}
\begin{proof}[Proof of \eqref{smtness estimate last y changing}]
Using Proposition \ref{esti of deriv of subordinate func}, as in the proof of \eqref{size estimate old}, we similarly come across the integral,
\begin{eqnarray}\label{int by u factor for the diff}
 | K_{\N}(x,y)-K_{\N}(x,y')|
  &\lesssim &  \int_0^\infty |\tau^k_xh_u(-y)-\tau^k_xh_u(-y')|\,\frac{du}{u}.
\end{eqnarray}
Now if $|y-y'|\leq \sqrt{u}$, using Proposition \ref{Dunkl heat ker esti} \ref{heat ii}, we get an estimate
\begin{eqnarray*}
 && |\tau^k_xh_u(-y)-\tau^k_xh_u(-y')|\\ &\lesssim & \frac{|y-y'|}{\sqrt{u}} \frac{1}{\mu_k(B(x, \sqrt{u}))}\,\Big(1+\frac{|x-y|}{\sqrt{u}}\Big)^{-2} \exp{(-c\,{d_G(x,\, y)^2}/{u})}\\
  &\leq & \frac{|y-y'|}{|x-y|} \frac{1}{\mu_k(B(x, \sqrt{u}))}\, \exp{(-c\,{d_G(x,\, y)^2}/{u})}.
\end{eqnarray*}
Also, when $|y - y'| > \sqrt{u}$, using Proposition \ref{Dunkl heat ker esti} \ref{heat i}, together with the fact that $|y - y'| < d_G(x,y)/2$ implies $|x - y| \sim |x - y'|$ and $d_G(x,y) \sim d_G(x,y')$; we obtain the same estimate
\begin{eqnarray*}
 && |\tau^k_xh_u(-y)-\tau^k_xh_u(-y')|\\ 
  &\lesssim &  \frac{1}{\mu_k(B(x, \sqrt{u}))}\,\Big(1+\frac{|x-y|}{\sqrt{u}}\Big)^{-2} \exp{(-c\,{d_G(x,\, y)^2}/{u})}\\
  &&+   \frac{1}{\mu_k(B(x, \sqrt{u}))}\,\Big(1+\frac{|x-y'|}{\sqrt{u}}\Big)^{-2} \exp{(-c\,{d_G(x,\, y')^2}/{u})}\\
  &\lesssim &  \frac{1}{\mu_k(B(x, \sqrt{u}))}\,\Big(1+\frac{|x-y|}{\sqrt{u}}\Big)^{-2} \exp{(-c\,{d_G(x,\, y)^2}/{u})}\\
  &\leq & \frac{|y-y'|}{\sqrt{u}} \frac{1}{\mu_k(B(x, \sqrt{u}))}\,\Big(1+\frac{|x-y|}{\sqrt{u}}\Big)^{-2} \exp{(-c\,{d_G(x,\, y)^2}/{u})}\\
  &\leq & \frac{|y-y'|}{|x-y|} \frac{1}{\mu_k(B(x, \sqrt{u}))}\, \exp{(-c\,{d_G(x,\, y)^2}/{u})}.
\end{eqnarray*}
Thus, from the above estimates and \eqref{int by u factor for the diff}, we can write
\begin{eqnarray*}
&& | K_{\N}(x,y)-K_{\N}(x,y')|\\
  &\lesssim & \frac{|y-y'|}{|x-y|} \int_0^\infty \frac{1}{\mu_k(B(x, \sqrt{u}))}\, \exp{(-c\,{d_G(x,\, y)^2}/{u})}\,\frac{du}{u}\\
  & = & \frac{|y-y'|}{|x-y|} \Big[ \int_0^{d_G(x,y)^2} \frac{1}{\mu_k(B(x, \sqrt{u}))}\, \exp{(-c\,{d_G(x,\, y)^2}/{u})}\,\frac{du}{u}\\
  && + \int_{d_G(x,y)^2}^\infty \frac{1}{\mu_k(B(x, \sqrt{u}))}\, \exp{(-c\,{d_G(x,\, y)^2}/{u})}\,\frac{du}{u}\,\Big].
  \end{eqnarray*}
The rest of the proof can be carried out in the same way as the proof of \eqref{size estimate old}.
\end{proof}
The next theorem is the main result of this section, which establishes the uniform boundedness of $\mathcal{T}_{\N}$.
\begin{thm}\label{bddness thm of partial sum}
The following are true.
\begin{enumerate}[label=(\roman*)]
    \item \label{TN p} If $1<p<\infty$ and $w$ is a $G$-invariant weight such that $w\in A^k_p$, then for any $f\in L^p(\mathbb{R}^d, w\,d\mu_k)$ we have the strong type inequality
    $$\Big(\int_{\mathbb{R}^d}|\mathcal{T}_{\N}f(x)|^p w(x)\, d\mu_k(x)  \Big)^{1/p}\leq C\, \Big(\int_{\mathbb{R}^d}|f(x)|^p w(x)\, d\mu_k(x)  \Big)^{1/p},$$
    where $C=C(d, k, G, s, p, w, \|\{v_j\}\|_{\ell^\infty(\mathbb{Z})})$. 

    \item \label{TN 1} If  $w$ is a $G$-invariant weight such that $w\in A^k_1$, then for any $f\in L^1(\mathbb{R}^d, w\,d\mu_k)$ we have the weak type inequality
    $$\sup\limits_{t>0}\,t\Big(\int\limits_{\{|\mathcal{T}_{\N}f(x)|>t\}} w(x)\, d\mu_k(x)  \Big)\leq C\, \int_{\mathbb{R}^d}|f(x)| w(x)\, d\mu_k(x)  ,$$
    where $C=C(d, k, G, s, w, \|\{v_j\}\|_{\ell^\infty(\mathbb{Z})})$. 

    \item \label{TN l infty} For any $f\in L^\infty(\mathbb{R}^d)$ we have the boundedness
    $$\|\mathcal{T}_{\N}f\|_{\text{BMO}_k} \leq C\, \|f\|_{L^\infty(d\mu_k)},$$
     where $C=C(d, k, G, s, \|\{v_j\}\|_{\ell^\infty(\mathbb{Z})})$. 

     \item \label{TN BMO} For any $f\in \text{BMO}_{G,k}(\mathbb{R}^d)$ we have the boundedness
    $$\|\mathcal{T}_{\N}f\|_{\text{BMO}_k} \leq C\, \|f\|_{\text{BMO}_{G,k}},$$
     where $C=C(d, k, G, s, \|\{v_j\}\|_{\ell^\infty(\mathbb{Z})})$. 
    \end{enumerate}
    It is to be noted that in each of the above cases, $C$ is independent of $\N$.
\end{thm}

\begin{proof}[Proof of \ref{TN p} and \ref{TN 1}] Follows directly from Proposition \ref{L2 bddness of TN}, Proposition \ref{kernel of TN}, and Theorem \ref{weighted calderon Zygmund bddness thm}. 
\end{proof}
\begin{proof}[Proof of \ref{TN l infty}] The proof will follow from \ref{TN BMO}, since we have
$$\|f\|_{\text{BMO}_{G,k}} \lesssim \|f\|_{L^\infty(d\mu_k)}.$$
\end{proof}
\begin{proof}[Proof of \ref{TN BMO}] 
We will closely follow the ideas developed in \cite{ChaoBODTFHSGBSO}. We first show that $\mathcal{T}_{\bf N}f(x)$ is finite almost everywhere for any 
$f \in \mathrm{BMO}_{G,k}(\mathbb{R}^d)$. Let $B=B(x_0, r)$ be a ball in $\mathbb{R}^d$ and $B^*=B(x_0, 5r)$. Then 
$$\mu_k(B^*) \sim \mu_k(B) \sim \mu_k(\mathcal{O}(B^*))\sim \mu_k(\mathcal{O}(B)).$$
Now
\begin{eqnarray*}
    f& = & \big(f-f_{\mathcal{O}(B)}\big)\chi_{\mathcal{O}(B^*)} + \big(f-f_{\mathcal{O}(B)}\big)\chi_{\mathcal{O}(B^*)^c} + f_{\mathcal{O}(B)}\\
    & = : & f_1+ f_2 + f_3.
\end{eqnarray*}
\vspace{.5cm}
{\bf Estimate for $f_1$:}
\vspace{.5cm}
\par We have 
\begin{eqnarray*}
    &&\int_{\mathbb{R}^d}|f_1(x)|\, d\mu_k(x)\\ &=& \int_{\mathcal{O}(B^*)}|\,f(x)-f_{\mathcal{O}(B)}|\, d\mu_k(x)\\
    &=& \int_{\mathcal{O}(B^*)}|\,f(x)-f_{\mathcal{O}(B^*)}|\, d\mu_k(x) + \mu_k(\mathcal{O}(B^*)) |\,f_{\mathcal{O}(B)}-f_{\mathcal{O}(B^*)}|\\
    &\lesssim & \mu_k(B^*) \big[\|f\|_{\text{BMO}_{G,k}} + \frac{1}{\mu_k(\mathcal{O}(B))} \int_{\mathcal{O}(B)}|\,f(y)-f_{\mathcal{O}(B^*)}|\, d\mu_k(y)\big]\\
    &\lesssim & \mu_k(B^*) \big[\|f\|_{\text{BMO}_{G,k}} + \frac{1}{\mu_k(\mathcal{O}(B^*))} \int_{\mathcal{O}(B^*)}|\,f(y)-f_{\mathcal{O}(B^*)}|\, d\mu_k(y)\big]\\
    &\lesssim & \mu_k(B^*)\|f\|_{\text{BMO}_{G,k}}.
\end{eqnarray*}
Then from Theorem \ref{bddness thm of partial sum} \ref{TN 1}, it follows that $\mathcal{T}_{\bf N}f_1$ is finite almost everywhere.

\vspace{.5cm}
{\bf Estimate for $f_2$:}
\vspace{.5cm}
\par From Proposition \ref{esti for fractional Dunkl heat kernel} \ref{frac heat i}, we can write 
\begin{eqnarray*}
 | \,e^{-t\,(-\Delta_k)^{s}}f_2(x)\,|&=& \big|\int_{\mathbb{R}^d}\tau^k_xh_{t,s}(-y)f_2(y) \, d\mu_k(y) \,\big| \\
 &=& \left|\int_{\mathcal{O}(B^*)^c}\tau^k_xh_{t,s}(-y)(f(y)-f_{\mathcal{O}(B)})\, d\mu_k(y) \right|\\
 &\leq & t\,\int_{d_G(x_0,y)>2r}\frac{|f(y)-f_{\mathcal{O}(B)}|}{\big(d_G(x,y)\big)^{d_k+2s}}\, d\mu_k(y)
\end{eqnarray*}
If we take $x\in B$, then we can write
\begin{eqnarray*}
 &&| \,e^{-t\,(-\Delta_k)^{s}}f_2(x)\,|\\
 &\leq & t\,\sum\limits_{j=1}^\infty\int\limits_{2^j r<d_G(x_0,y)\leq 2^{j+1}r}\frac{|f(y)-f_{\mathcal{O}(B)}|}{\big(d_G(x,y)\big)^{d_k+2s}}\, d\mu_k(y) \\  
 &\leq & t\,\sum\limits_{j=1}^\infty (2^{j-1}r)^{-d_k-2s}\int\limits_{d_G(x_0,y)\leq 2^{j+1}r}|f(y)-f_{\mathcal{O}(2^{j+1}B)}+f_{\mathcal{O}(2^{j+1}B)}-f_{\mathcal{O}(B)}|\, d\mu_k(y) \\ 
 &\lesssim & t\,\sum\limits_{j=1}^\infty (2^{j}r)^{-d_k-2s} \Big[\int_{\mathcal{O}(2^{j+1}B)}|f(y)-f_{\mathcal{O}(2^{j+1}B)}|\, d\mu_k(y)\\
 && +\mu_k(\mathcal{O}(2^{j+1}B))|f_{\mathcal{O}(B)}-f_{\mathcal{O}(2^{j+1}B)}|\Big] \\
 &\lesssim & t\,\sum\limits_{j=1}^\infty (2^{j}r)^{-d_k-2s} \mu_k(\mathcal{O}(2^{j+1}B))\Big[ \|f\|_{\text{BMO}_{G,k}} + |f_{\mathcal{O}(B)}-f_{\mathcal{O}(2^{j+1}B)}|\Big]\\
 &\lesssim & t\,\sum\limits_{j=1}^\infty (2^{j}r)^{-d_k-2s} \mu_k(2^{j}B)\Big[ \|f\|_{\text{BMO}_{G,k}} + \sum\limits_{n=1}^j|f_{\mathcal{O}(2^{n+1}B)}-f_{\mathcal{O}(2^{\,n}B)}|\Big]\\
 &\lesssim & t\,\sum\limits_{j=1}^\infty (2^{j}r)^{-d_k-2s} 2^{jd_k}\mu_k(B)\Big[ \|f\|_{\text{BMO}_{G,k}} + \sum\limits_{n=1}^j2\,\|f\|_{\text{BMO}_{G,k}}\Big]\\
 &\lesssim & t\,r^{-d_k-2s}\mu_k(B)\, \|f\|_{\text{BMO}_{G,k}}\sum\limits_{j=1}^\infty 2^{-2js}(1+2j)\\
  &\lesssim & t\,r^{-d_k-2s}\mu_k(B)\, \|f\|_{\text{BMO}_{G,k}}.
\end{eqnarray*}
Since $x_0$ and $r$ are arbitrary, and $\mathcal{T}_{\N}f_2(x)$ is a finite sum of terms of the form $e^{-t(-\Delta_k)^s}f_2(x)$, it follows from the above that $\mathcal{T}_{\N}f_2$ is finite almost everywhere.

\vspace{.5cm}
{\bf Estimate for $f_3$:}
\vspace{.5cm}

We note that for any $t>0$,
$$\int_{\mathbb{R}^d}\tau^k_xh_{t,s}(-y)\,d\mu_k(y)=\int_{\mathbb{R}^d}h_{t,s}(y)\,d\mu_k(y)=1,$$
and hence $\mathcal{T}_{\N}f_3(x)=0.$
\par Combining the estimates for $f_1$, $f_2$, and $f_3$, we can finally conclude that $\mathcal{T}_{\N}f$ is finite almost everywhere.

\par Now to prove the boundedness from $\text{BMO}_{G,k}(\mathbb{R}^d)$ to $\text{BMO}_{k}(\mathbb{R}^d)$, let us consider an arbitrary ball $B=B(x_0,r)\subseteq \mathbb{R}^d$, and choose $x_1\in B$ such that $|\mathcal{T}_{\N}f_2(x_1)|<\infty$, where the decomposition of $f$ into $f_1, f_2$, and $f_3$ is the same as earlier in the proof, but this time with respect to the new $B$. Also, let $B^*=B(x_0, 5r)$ and we choose $c_B=\mathcal{T}_{\N}f_2(x_1)$. Then
\begin{eqnarray*}
&& \frac{1}{\mu_k(B)}\int_B\big| \mathcal{T}_{\N}f(x)- c_B\big|\, d\mu_k(x)\\ 
&\leq & \frac{1}{\mu_k(B)}\int_B\big| \mathcal{T}_{\N}f_1(x)\big|\, d\mu_k(x)+ \frac{1}{\mu_k(B)}\int_B\big| \mathcal{T}_{\N}f_2(x)- \mathcal{T}_{\N}f_2(x_1)\big|\, d\mu_k(x) \\
&=:& I+II.
\end{eqnarray*}
The estimate of the first term follows from Proposition \ref{L2 bddness of TN} and \cite[p. 46]{JiuTDOTHSAWTDO} as follows.
\begin{eqnarray*}
  I &\leq & \frac{\big(\mu_k(B)\big)^{1/2}}{\mu_k(B)} \|\mathcal{T}_{\N}f_1\|_{L^2(d\mu_k)} \\
  &\leq & \big(\mu_k(B)\big)^{-1/2}\|f_1\|_{L^2(d\mu_k)} \\
  &=& \big(\mu_k(B)\big)^{-1/2}\Big( \int_{\mathcal{O}(B^*)}|f(x)- f_{\mathcal{O}(B)}|^2 \,d\mu_k(x)\Big)^{1/2} \\
  &\lesssim& \big(\mu_k(B)\big)^{-1/2}\Big( \int_{\mathcal{O}(B^*)}|f(x)- f_{\mathcal{O}(B^*)}|^2 \,d\mu_k(x) +\!\!\! \int_{\mathcal{O}(B^*)}|f_{\mathcal{O}(B)}- f_{\mathcal{O}(B^*)}|^2 \,d\mu_k(x)\Big)^{1/2} \\
  &\lesssim& \big(\mu_k(B)\big)^{-1/2}\Big( \|f\|_{\text{BMO}_{G,k}}^2 \,\mu_k(\mathcal{O}(B^*)) +|f_{\mathcal{O}(B)}- f_{\mathcal{O}(B^*)}|^2 \,\mu_k(\mathcal{O}(B^*))\Big)^{1/2} \\
  &\lesssim & \big(\mu_k(B)\big)^{-1/2}\Big( \|f\|_{\text{BMO}_{G,k}}^2 \,\mu_k(\mathcal{O}(B^*))+ \int_{\mathcal{O}(B^*)}|f(y)- f_{\mathcal{O}(B^*)}|^2\, d\mu_k(y)\Big)^{1/2}\\
&\lesssim & \|f\|_{\text{BMO}_{G,k}}.
\end{eqnarray*}
On the other hand, for the second term, we use \eqref{smtness estimate last x changing} to obtain
\begin{eqnarray*}
  II &\leq &  \frac{1}{\mu_k(B)} \int_{B}\int_{\mathbb{R}^d}|K_{\N}(x,y)- K_{\N}(x_1, y)|\,|f_2(y)|\, d\mu_k(y)d\mu_k(x)\\
  &\leq &  \frac{\|\{v_j\}\|_{\ell^\infty(\mathbb{Z})}}{\mu_k(B)} \int_{B}\int_{\mathcal{O}(B^*)^c}\frac{1}{\mu_k\big(B(x, d_G(x,y))\big)}\,\frac{|x-x_1|}{|x-y|}\,|f(y)-f_{\mathcal{O}(B)}|\, d\mu_k(y)d\mu_k(x)\\
  &\lesssim &  \frac{r}{\mu_k(B)} \int_{B}\Big[\sum\limits_{j=1}^\infty\int\limits_{2^{j}r< d_G(x_0,y)\leq 2^{j+1}r}\frac{|f(y)-f_{\mathcal{O}(B)}|}{|x-y|\,\mu_k\big(B(x, d_G(x,y))\big)}\, d\mu_k(y)\Big]\, d\mu_k(x)\\
  &\lesssim &  \frac{r}{\mu_k(B)} \int_{B}\Big[\sum\limits_{j=1}^\infty  \big( 2^{j-1}r\big)^{-1}\frac{1}{\mu_k\big(B(x, 2^{j-1}r \big)}
  \int\limits_{ d_G(x_0,y)\leq 2^{j+1}r}\!\!\!|f(y)-f_{\mathcal{O}(B)}|\, d\mu_k(y)\Big]\, d\mu_k(x)\\
  &\lesssim &  \frac{r}{\mu_k(B)} \int_{B}\Big[\sum\limits_{j=1}^\infty  \big( 2^{j-1}r\big)^{-1}\frac{1}{\mu_k\big(B(x, 2^{j-1}r \big)}
  \int\limits_{ d_G(x,y)\leq 2^{j+2}r}|f(y)-f_{\mathcal{O}(B)}|\, d\mu_k(y)\Big]\, d\mu_k(x)\\
  &\lesssim &  \frac{1}{\mu_k(B)} \int_{B}\Big[\sum\limits_{j=1}^\infty 2^{-j}\frac{1}{\mu_k\big(B(x, 2^{j+2}r \big)}
  \int\limits_{\mathcal{O}( B(x, 2^{j+2}r )}\!\!\!|f(y)-f_{\mathcal{O}(B)}|\, d\mu_k(y)\Big]\, d\mu_k(x)\\
   &\lesssim &  \frac{1}{\mu_k(B)} \int_{B}\Big[\sum\limits_{j=1}^\infty 2^{-j}\frac{1}{\mu_k\big(\mathcal{O}( B(x, 2^{j+2}r ) \big)}
  \int\limits_{\mathcal{O}( B(x, 2^{j+2}r )}|f(y)-f_{\mathcal{O}(B)}|\, d\mu_k(y)\Big]\, d\mu_k(x)\\
&\lesssim & \frac{1}{\mu_k(B)}    \|f\|_{\text{BMO}_{G,k}}\int_{B}\Big[\sum\limits_{j=1}^\infty 2^{-j} (1+2j)\Big]\, d\mu_k(x)\\
&\lesssim &    \|f\|_{\text{BMO}_{G,k}}.
\end{eqnarray*}
This completes the proof of \ref{TN BMO}.
\end{proof}


\section{Proofs of the main results}\label{Sec-proof of main results} 
In the final section of the paper, we first study some properties of lacunary sequences and prove a Cotlar-type inequality, and then proceed with the proofs of all the main results stated in Section \ref{Sec-intro}.

\subsection{Properties of lacunary sequences}\label{subsec- prop lacu sequ}
We prove the following property of the $\ell$-lacunary sequence, which allows us to assume, without loss of generality, that
$$1<\ell\leq \frac{a_{j+1}}{a_j}\leq \ell^{\,2}, \qquad j\in \mathbb{Z}.$$
\begin{lem}
 Let $\ell>1$, $\{a_j\}_{j\in \mathbb{Z}}$ an $\ell$-lacunary sequence, and $\{v_j\}_{j\in \mathbb{Z}}\in \ell^\infty (\mathbb{Z})$.  Then there exist an $\ell$-lacunary sequence $\{b_j\}_{j\in \mathbb{Z}}$ be  and $\{\omega_j\}_{j\in \mathbb{Z}}\in \ell^\infty (\mathbb{Z})$ such that
\begin{enumerate}[label=(\roman*)]
\item \leavevmode\vspace*{-\dimexpr\abovedisplayskip + \baselineskip} $$1<\ell\leq \frac{b_{j+1}}{b_j}\leq \ell^{\,2}, \quad j\in \mathbb{Z} \text{ and } \|\{\omega_j\}\|_{\ell^\infty(\mathbb{Z})}=\|\{v_j\}\|_{\ell^\infty(\mathbb{Z})}$$
\item For any $\N=(N_1, N_2)\in \mathbb{Z}^2$, there exists $\N'=(N'_1, N'_2)\in \mathbb{Z}^2$ such that $\mathcal{T}_{\N}=\widetilde{\mathcal{T}}_{\N'}$, where
$$\widetilde{\mathcal{T}}_{\N'}f(x)=\sum\limits_{j=N'_1}^{N'_2} \omega_j \Big(e^{-b_{j+1}(-\Delta_k)^s}f(x) - e^{-b_j(-\Delta_k)^s} f(x)\Big).$$
\end{enumerate}
\end{lem}
\begin{proof}
 The proof follows along the lines of \cite[Proposition 3.2]{BernardisDTIWS}. 
Let $b_0=a_0$. Now, we already have $a_1/a_0\ge\ell$.
If $\ell^2\ge a_1/a_0$, define $b_1=a_1$; otherwise,
define $b_1=\ell a_0$. In the second case, clearly
$b_1/b_0=\ell\le\ell^2$ and
$$
    \frac{a_1}{b_1}
    \ge \frac{\ell^2a_0}{\ell a_0}
    =\ell.
$$
Now, if $\ell^2\ge a_1/b_1$, define $b_2=a_1$;
otherwise, define $b_2=\ell^2a_0$. Since $\ell>1$,
this process ends at some $j_0$ such that $b_{j_0}=a_1$.
The construction on the remaining intervals
$[a_j,a_{j+1}]$, in both directions, can be carried out
in a similar way. Thus, we obtain a sequence
$\{b_j\}_{j\in\mathbb{Z}}$ containing every $a_j$ and satisfying
$$1<\ell\le\frac{b_{j+1}}{b_j}\le\ell^2,
    \qquad j\in\mathbb{Z}.$$
Now let $J(j)=\{n:a_j\le b_n<a_{j+1}\}$, and define $\omega_n=v_j$ if $n\in J(j)$.
Since the sets $J(j)$ are nonempty and form a partition
of $\mathbb{Z}$, we have
$$\|\{\omega_j\}\|_{\ell^\infty(\mathbb{Z})}
    =
    \|\{v_j\}\|_{\ell^\infty(\mathbb{Z})}.$$
Then, by telescoping, we are able to write
$$
\begin{aligned}
    &v_j\Bigl(
        e^{-a_{j+1}(-\Delta_k)^s}f(x)
        -e^{-a_j(-\Delta_k)^s}f(x)
    \Bigr)\\
    &\qquad=
    \sum_{n\in J(j)}
    \omega_n\Bigl(
        e^{-b_{n+1}(-\Delta_k)^s}f(x)
        -e^{-b_n(-\Delta_k)^s}f(x)
    \Bigr).
\end{aligned}
$$
Let $\mathbf{N}'=(N_1',N_2')\in\mathbb{Z}^2$ be such that $b_{N_1'}=a_{N_1}$ and $b_{N_2'+1}=a_{N_2+1}$, so that
$$
\begin{aligned}
    \mathcal{T}_{\mathbf{N}}f(x)
    &=
    \sum_{n=N_1'}^{N_2'}
    \omega_n\Bigl(
        e^{-b_{n+1}(-\Delta_k)^s}f(x)
        -e^{-b_n(-\Delta_k)^s}f(x)
    \Bigr)\\
    &=\widetilde{\mathcal{T}}_{\mathbf{N}'}f(x).
\end{aligned}
$$
\end{proof}

\subsection{Cotlar-type inequality and boundedness of the maximal operator}\label{subsec- cotlar ineq and bddness of max op}
The main result of this section is a Cotlar-type inequality for the maximal operator $\mathcal{T}^*$, which serves as the main ingredient in the proof of Theorem \ref{main thm for maximal op}. Before stating this inequality, we define two maximal functions that will be used in the inequality.
$$\mathcal{M}f(x):=\sup_{r>0} \frac{1}{\mu_k(B(x,r))}\int_{B(x,r)} |f(y)|\, d\mu_k(y)
$$ and for $1<q<\infty$, 
$$\mathcal{M}_qf(x):=\sup_{r>0} \Big(\frac{1}{\mu_k(B(x,r))}\int_{B(x,r)} |f(y)|^q\, d\mu_k(y)\Big)^{1/q},
$$ 
for a suitable function $f$ on $\mathbb{R}^d$. Clearly, by H\"older's inequality 
$\mathcal{M}f(x)\leq\mathcal{M}_qf(x)$.

\par The Cotlar-type inequality associated with the operator $\mathcal{T}^*$ is stated and proved in the following theorem.
\begin{thm}\label{cotlar ineq thm}
For any $1<q<\infty$, there exists  $C=C(d, k, G, s, \|\{v_j\}\|_{\ell^\infty(\mathbb{Z})}, \ell, q)>0$ such that for any $x\in \mathbb{R}^d$ and $M\in \mathbb{N}$, we have
$$\mathcal{T}_M^*f(x) \leq C\,\Big( \mathcal{M}(\mathcal{T}_{(-M,\,M)}f)(x) + \sum \limits_{\sigma \in G}\mathcal{M}_qf(\sigma(x))\Big),$$
where
$$\mathcal{T}_M^*f(x)=\sup\limits_{-M\leq N_1< N_2\leq M} \big|\mathcal{T}_{\N} f(x)\big|.$$
\end{thm}
\begin{proof}
For any $x\in \mathbb{R}^d$ and $\N=(N_1, N_2)$, we can write 
$$\mathcal{T}_{\N} f(x)=\mathcal{T}_{(N_1,\,M)}f(x)- \mathcal{T}_{(N_2+1,\,M)}f(x),$$
where $-M\leq N_1< N_2\leq M$. Therefore, it is enough to prove the theorem for $\mathcal{T}_{(m,\,M)}$, where $|m|<M$.
\par Let $B_n :=B(x, a^{\frac{1}{2s}}_n)$ for any $n \in \mathbb{Z}$ and 
\begin{eqnarray*}
  f=f \chi_{\mathcal{O}(B_m)}  + f \chi_{\mathcal{O}(B_m)^c}=: f_1+ f_2.  
\end{eqnarray*}
Hence, we can split
\begin{eqnarray*}
  |\mathcal{T}_{(m,\,M)}f(x)|& \leq &  |\mathcal{T}_{(m,\,M)}f_1(x)| +  |\mathcal{T}_{(m,\,M)}f_2(x)| \\
  &=:& I +II.
\end{eqnarray*}
\vspace{.5cm}
{\bf Estimate for $I$:}
\vspace{.5cm}

To estimate $I$, we first express it as 
\begin{eqnarray}\label{just one time use 1}
I &=&\Big|\int_{\mathbb{R}^d}\sum\limits_{j=m}^{M} v_j \Big(\tau^k_xh_{a_{j+1},s}(-y) - \tau^k_xh_{a_j,s}(-y) \Big)f_1(y)\, d\mu_k(y) \Big| \\
&\leq & \|\{v_j\}\|_{\ell^\infty(\mathbb{Z})}\int_{\mathbb{R}^d}\sum\limits_{j=m}^{M} \big|\tau^k_xh_{a_{j+1},s}(-y) - \tau^k_xh_{a_j,s}(-y) \big|\,|f_1(y)|\, d\mu_k(y)  \nonumber.
\end{eqnarray}
Now using Proposition \ref{esti for fractional Dunkl heat kernel} \ref{frac heat i}, for any $j \in \mathbb{Z}$ we get an estimate
\begin{eqnarray*}
\big|\tau^k_xh_{a_{j+1},s}(-y) - \tau^k_xh_{a_j,s}(-y) \big|& \leq & \tau^k_xh_{a_{j+1},s}(-y) +\tau^k_xh_{a_j,s}(-y) \\
& \lesssim & \frac{a_{j+1}}{\mu_k\big( B(x, a_{j+1}^{\frac{1}{2s}}\big)} \times \frac{1}{a_{j+1}} + \frac{a_{j}}{\mu_k\big( B(x, a_{j}^{\frac{1}{2s}}\big)} \times \frac{1}{a_{j}}\\
& \lesssim & \frac{1}{\mu_k\big( B(x, a_{j}^{\frac{1}{2s}}\big)}.
\end{eqnarray*}
Substituting this estimate into \eqref{just one time use 1}, we obtain the desired estimate for $I$ as follows.
\begin{eqnarray*}
 I &\lesssim &  \int_{\mathbb{R}^d}\sum\limits_{j=m}^{M}  \frac{1}{\mu_k\big( B(x, a_{j}^{\frac{1}{2s}}\big)}|f_1(y)|\, d\mu_k(y)\\  
 &= &  \sum\limits_{j=m}^{M}  \frac{1}{\mu_k\big( B(x, a_{j}^{\frac{1}{2s}}\big)}\int_{\mathcal{O}(B_m)}|f(y)|\, d\mu_k(y)\\  
 &\lesssim &    \frac{1}{\mu_k\big( B(x, a_{m}^{\frac{1}{2s}}\big)}\int_{\mathcal{O}(B_m)}|f(y)|\, d\mu_k(y)\sum\limits_{j=m}^{M} \Big(\frac{a_m}{a_j}\Big)^{\frac{d}{2s}}\\  
 &\leq & \sum \limits_{\sigma \in G}
    \frac{1}{\mu_k\big( B(\sigma(x), a_{m}^{\frac{1}{2s}}\big)}\int_{B_m}|f(\sigma(y))|\, d\mu_k(y)\sum\limits_{j=m}^{M} \ell^{\frac{-d}{2s}(-m+j)}\\  
    &\leq & \sum \limits_{\sigma \in G}\mathcal{M}f(\sigma(x))\sum\limits_{j=m}^{\infty}\ell^{\frac{-d}{2s}(-m+j)}\\
    &\lesssim & \sum \limits_{\sigma \in G}\mathcal{M}f(\sigma(x)).
\end{eqnarray*}

\vspace{.5cm}
{\bf Estimate for $II$:}
\vspace{.5cm}

To obtain the required estimates for the term $II$, we decompose it into four parts as follows.
\begin{eqnarray*}
  II &=& |\mathcal{T}_{(m,\,M)}f_2(x)|\\
  &=& \frac{1}{\mu_k\big(B(x, \frac{1}{2}a_{m-1}^{\frac{1}{2s}})\big)} \int_{B(x, \frac{1}{2}a_{m-1}^{\frac{1}{2s}})}|\mathcal{T}_{(m,\,M)}f_2(x)| \, d\mu_k(z)\\
  &\leq & \frac{1}{\mu_k\big(B(x, \frac{1}{2}a_{m-1}^{\frac{1}{2s}})\big)} \int_{B(x, \frac{1}{2}a_{m-1}^{\frac{1}{2s}})}|\mathcal{T}_{(-M,\,M)}f(z)| \, d\mu_k(z) \\
  && + \frac{1}{\mu_k\big(B(x, \frac{1}{2}a_{m-1}^{\frac{1}{2s}})\big)} \int_{B(x, \frac{1}{2}a_{m-1}^{\frac{1}{2s}})}|\mathcal{T}_{(-M,\,M)}f_1(z)| \, d\mu_k(z) \\
  && + \frac{1}{\mu_k\big(B(x, \frac{1}{2}a_{m-1}^{\frac{1}{2s}})\big)} \int_{B(x, \frac{1}{2}a_{m-1}^{\frac{1}{2s}})}|\mathcal{T}_{(m,\,M)}f_2(z)-\mathcal{T}_{(m,\,M)}f_2(x)| \, d\mu_k(z) \\
  && + \frac{1}{\mu_k\big(B(x, \frac{1}{2}a_{m-1}^{\frac{1}{2s}})\big)} \int_{B(x, \frac{1}{2}a_{m-1}^{\frac{1}{2s}})}|\mathcal{T}_{(-M,\,m-1)}f_2(z)| \, d\mu_k(z) \\
  &=:& A_1 + A_2+ A_3+ A_4,
\end{eqnarray*}
where $A_4=0$ if $m+1=-M$.
\par We estimate each of the four terms separately. Clearly,
$$A_1 \leq \mathcal{M}(\mathcal{T}_{(-M,\,M)}f)(x).$$

For $A_2$, by H\"older's inequality, Theorem \ref{bddness thm of partial sum} \ref{TN p}, and the triangle inequality, give 
\begin{eqnarray*}
    A_2 &\leq & \frac{1}{\mu_k\big(B(x, \frac{1}{2}a_{m-1}^{\frac{1}{2s}})\big)} \Big(\mu_k\big(B(x, \frac{1}{2}a_{m-1}^{\frac{1}{2s}})\big)\Big)^{1/q'}\Big(\int_{\mathbb{R}^d}|\mathcal{T}_{(-M,\,M)}f_1(z)|^q \, d\mu_k(z) \Big)^{1/q}\\
    &\lesssim & \frac{1}{\big(\mu_k\big(B(x, \frac{1}{2}a_{m-1}^{\frac{1}{2s}})\big)\big)^{1/q}}\Big(\int_{\mathbb{R}^d}|f_1(z)|^q \, d\mu_k(z) \Big)^{1/q}\\
    &\leq & \Big(\frac{1}{\mu_k\big(B(x, \frac{1}{2}a_{m-1}^{\frac{1}{2s}})\big)}\int_{\mathcal{O}(B_m)}|f(z)|^q \, d\mu_k(z) \Big)^{1/q}\\
    &\lesssim & \Big(\sum\limits_{\sigma \in G}\frac{1}{\mu_k\big(B(\sigma(x), a_{m-1}^{\frac{1}{2s}})\big)}\int_{B_m}|f(\sigma(z))|^q \, d\mu_k(z) \Big)^{1/q}\\
    &\lesssim & \sum\limits_{\sigma \in G}\mathcal{M}_qf(\sigma(x)).
\end{eqnarray*}
We note that for $z\in B(x, \frac{1}{2}a^{\frac{1}{2s}}_{m-1})$, \eqref{smtness estimate last x changing} yields
\begin{eqnarray}\label{one time use 2}
 && |\mathcal{T}_{(m,\,M)}f_2(z)-\mathcal{T}_{(m,\,M)}f_2(x)| \\ 
 &\leq&  \int_{\mathbb{R}^d} |K_{(m, \, M)}(z,y)-K_{(m, \, M)}(x,y)| \, |f_2(y)|\, d\mu_k(y) \nonumber\\
 & \lesssim &  \int_{\mathcal{O}(B_m)^c} \frac{1}{\mu_k\big( B(z, d_G(z,y))\big)} \frac{|z-x|}{|z-y|}\, |f(y)|\, d\mu_k(y) \nonumber\\
 &\lesssim &  \frac{1}{2}a^{\frac{1}{2s}}_{m-1}\int_{\mathcal{O}(B_m)^c} \frac{1}{\mu_k\big( B(z, d_G(z,y))\big)} \frac{1}{|z-y|}\, |f(y)|\, d\mu_k(y) \nonumber\\
 &= &  \frac{1}{2}a^{\frac{1}{2s}}_{m-1}\sum\limits_{j=m}^\infty \int\limits_{a^{\frac{1}{2s}}_{j}<d_G(x,y) \leq a^{\frac{1}{2s}}_{j+1}} \frac{1}{\mu_k\big( B(z, d_G(z,y))\big)} \frac{1}{|z-y|}\, |f(y)|\, d\mu_k(y) \nonumber\\
 &\lesssim &  a^{\frac{1}{2s}}_{m-1}\sum\limits_{j=m}^\infty \int\limits_{d_G(x,y) \leq a^{\frac{1}{2s}}_{j+1}} \frac{2}{a^{\frac{1}{2s}}_{j}\,\mu_k\big( B(z, \frac{1}{2}a^{\frac{1}{2s}}_{j})\big)}  |f(y)|\, d\mu_k(y) \nonumber\\
 &\lesssim &  a^{\frac{1}{2s}}_{m-1}\sum\limits_{j=m}^\infty \int\limits_{d_G(x,y) \leq a^{\frac{1}{2s}}_{j+1}} \frac{1}{a^{\frac{1}{2s}}_{j}\,\mu_k\big( B(z, a^{\frac{1}{2s}}_{j})\big)}  |f(y)|\, d\mu_k(y) \nonumber\\
 &\lesssim & a^{\frac{1}{2s}}_{m-1}\sum\limits_{j=m}^\infty  \frac{1}{a^{\frac{1}{2s}}_{j}\,\mu_k\big( B(x, \frac{1}{2}a^{\frac{1}{2s}}_{j})\big)} \int_{\mathcal{O}(B_{j+1})} |f(y)|\, d\mu_k(y) \nonumber\\
 &\lesssim & \sum\limits_{\sigma \in G}\mathcal{M}f(\sigma(x))\sum\limits_{j=m}^\infty \Big(\frac{a_{m-1}}{a_j}\Big)^{\frac{1}{2s}} \nonumber\\
 &\lesssim & \sum\limits_{\sigma \in G}\mathcal{M}f(\sigma(x))\sum\limits_{j=m}^\infty \ell^{\frac{m-j-1}{2s}}\nonumber\\
 &\lesssim & \sum\limits_{\sigma \in G}\mathcal{M}f(\sigma(x)),\nonumber
 \end{eqnarray}
 where the domination by the maximal function follows in the same manner as in the estimate for $A_2$.
\par Now, from \eqref{one time use 2}, the estimate for $A_3$ follows immediately:
\begin{eqnarray*}
    A_3 &=& \frac{1}{\mu_k\big(B(x, \frac{1}{2}a_{m-1}^{\frac{1}{2s}})\big)} \int_{B(x, \frac{1}{2}a_{m-1}^{\frac{1}{2s}})}|\mathcal{T}_{(m,\,M)}f_2(z)-\mathcal{T}_{(m,\,M)}f_2(x)| \, d\mu_k(z) \\&\lesssim & \sum\limits_{\sigma \in G}\mathcal{M}f(\sigma(x)).
\end{eqnarray*}
For the estimate of $A_4$, we first use the mean value theorem to deduce that
\begin{eqnarray*}
&& |\mathcal{T}_{(-M,\,m-1)}f_2(z)|\\
 &\lesssim & \int_{\mathbb{R}^d}\sum\limits_{j=-M}^{m-1} \big|\tau^k_zh_{a_{j+1},s}(-y) - \tau^k_zh_{a_j,s}(-y) \big|\, |f_2(y)|\, d\mu_k(y)\\
 &\lesssim & \int_{\mathcal{O}(B_m)^c}\sum\limits_{j=-M}^{m-1} (a_{j+1}-a_j)\big|\frac{\partial}{\partial t}\tau^k_zh_{t,s}(-y) \big|_{t=b_j}\, |f(y)|\, d\mu_k(y)\\
 &= & \sum \limits_{n=m}^\infty\int\limits_{a_n^{\frac{1}{2s}}\leq d_G(x, y)<a_{n+1}^{\frac{1}{2s}}}\sum\limits_{j=-M}^{m-1} (a_{j+1}-a_j)\big|\frac{\partial}{\partial t}\tau^k_zh_{t,s}(-y) \big|_{t=b_j}\, |f(y)|\, d\mu_k(y),
\end{eqnarray*}
where $a_j\leq b_j \leq a_{j+1}$. 
Now, to complete the estimate of $A_4$, we apply Proposition \ref{esti for fractional Dunkl heat kernel} \ref{frac heat ii} and the fact that $z\in B(x, \frac{1}{2}a_{m-1}^{\frac{1}{2s}})$ to conclude that
\begin{eqnarray*}
    &&|\mathcal{T}_{(-M,\,m-1)}f_2(z)|\\
 &\lesssim & \sum \limits_{n=m}^\infty\int\limits_{a_n^{\frac{1}{2s}}\leq d_G(x, y)<a_{n+1}^{\frac{1}{2s}}}\sum\limits_{j=-M}^{m-1} (\ell^{\, 2}-1)a_j\,\Big[ \frac{1}{\mu_k \big(B(z, d_G(z,y))\big)} \frac{1}{d_G(z,y)^{2s}} \Big]_{t=b_j}\, |f(y)|\, d\mu_k(y)\\
  &\lesssim & \sum \limits_{n=m}^\infty\int\limits_{a_n^{\frac{1}{2s}}\leq d_G(x, y)<a_{n+1}^{\frac{1}{2s}}}\sum\limits_{j=-M}^{m-1} \frac{a_j}{a_n}\, \frac{1}{\mu_k \big(B(z, \frac{1}{2}a_n^{\frac{1}{2s}})\big)} \, |f(y)|\, d\mu_k(y)\\
   &\lesssim & \sum \limits_{n=m}^\infty\int\limits_{d_G(x, y)<a_{n+1}^{\frac{1}{2s}}} \frac{1}{\mu_k \big(B(z, a_n^{\frac{1}{2s}})\big)} \, |f(y)|\, d\mu_k(y)\,\sum\limits_{j=-M}^{m-1} \frac{a_j}{a_n}\\
   &\lesssim & \sum \limits_{n=m}^\infty \frac{1}{\mu_k \big(B(x, \frac{1}{2}a_n^{\frac{1}{2s}})\big)} \int\limits_{\mathcal{O}(B_{n+1})}  |f(y)|\, d\mu_k(y)\,\sum\limits_{j=-M}^{m-1} \ell^{\, j-n}\\
   &\lesssim & \sum\limits_{\sigma \in G}\mathcal{M}f(\sigma(x)) \sum \limits_{n=m}^\infty \ell^{-(n-m+1)}\\
    &\lesssim & \sum\limits_{\sigma \in G}\mathcal{M}f(\sigma(x)).
\end{eqnarray*}
Hence, the required estimate for $A_4$ follows in a similar way to that of $A_3$.
\end{proof}

\subsection{Proof of the boundedness of the maximal operator}\label{subsec- proof of bddness of maximal op}
Having proved the Cotlar-type inequality and Theorem \ref{bddness thm of partial sum}, we are now ready to prove our main theorem concerning the boundedness of the maximal operator.

\begin{proof}[Proof of Theorem \ref{main thm for maximal op} \ref{T star p}]
 We already know that $\mathcal{M}$ is bounded on $L^p(\mathbb{R}^d, wd\mu_k)$ (see \cite[p.5]{StrombergTorchinskybook}). Moreover, one can choose $1<q<p$ such that $w\in A^k_{p/q}$ (see \cite[p.10]{StrombergTorchinskybook}). It is then well known that $\mathcal{M}_q$ is also bounded on $L^p(\mathbb{R}^d, wd\mu_k)$. Since we have the inequality in Theorem \ref{cotlar ineq thm}, the operators $T_{\N}$ are uniformly bounded on $L^p(\mathbb{R}^d, wd\mu_k)$, and since both the measure and the weights are $G$-invariant, it follows that
 $$\Big(\int_{\mathbb{R}^d}(\mathcal{T}_M^*f(x))^p w(x)\, d\mu_k(x)  \Big)^{1/p}\leq C\, \Big(\int_{\mathbb{R}^d}|f(x)|^p w(x)\, d\mu_k(x)  \Big)^{1/p}.$$
 As $C$ does not depend on $M$, letting $M\to \infty$ completes the proof.
\end{proof}
\begin{proof}[Proof of Theorem \ref{main thm for maximal op} \ref{T star 1}]
 Let $\mathcal{T}f(x)= \{\mathcal{T}_{\N}f(x)\}_{\N\in \mathbb{Z}^2}$. Then $\mathcal{T}$ is an $\ell^\infty(\mathbb{Z}^2)$-valued operator. Now,
$$\|\mathcal{T}f(x)\|_{\ell^\infty(\mathbb{Z}^2)}=\mathcal{T}^*f(x),$$
and hence the first part of this theorem implies that $\mathcal{T}$ maps $L^p(\mathbb{R}^d, w\,d\mu_k)$ boundedly into $L^p(\mathbb{R}^d,\ell^\infty(\mathbb{Z}^2), w\,d\mu_k)$ for $1<p<\infty$ and $G$-invariant weights $w\in A^k_p$. Also, the vector-valued kernel associated with $\mathcal{T}$ is given by
$$\mathcal{K}(x,y)=\{K_{\N}(x,y)\}_{\N\in \mathbb{Z}^2},$$
which obviously satisfies the estimates \eqref{size estimate old}--\eqref{smtness estimate last x changing}, with the usual norm replaced by the $\ell^\infty(\mathbb{Z}^2)$ norm. So, by a straightforward extension of the Dunkl--Calder\'on--Zygmund theory to the vector-valued case, we obtain the proof of \ref{T star 1}.
\end{proof}
\begin{proof}[Proof of Theorem \ref{main thm for maximal op} \ref{T star l infty}]
  The proof will follow from  Theorem \ref{main thm for maximal op} \ref{T star BMO}, since we have
$$\|f\|_{\text{BMO}_{G,k}} \lesssim \|f\|_{L^\infty(d\mu_k)}.$$  
\end{proof}
 \begin{proof}[Proof of Theorem \ref{main thm for maximal op} \ref{T star BMO}]

Let $f\in\mathrm{BMO}_{G,k}(\mathbb{R}^d)$, and set
$$E=\{x\in\mathbb{R}^d:\mathcal{T}^*f(x)<\infty\}.$$
If $\mu_k(E)=0$, then $\mathcal{T}^*f=\infty$ $\mu_k$-almost
everywhere, and there is nothing to prove.
Suppose, therefore, that $\mu_k(E)>0$.
Choose $R_0>0$ such that
$\mu_k\bigl(E\cap B(0,R_0)\bigr)>0$. Fix $R\ge R_0$, and let $B=B(0,R)$.
Writing $4B=B(0,4R)$, we decompose
$$f=f_1+f_2+c_B,$$
where $c_B=f_{4B}$, $f_1=(f-c_B)\chi_{4B}$, and $f_2=(f-c_B)\chi_{(4B)^c}$. Since every ball centred at the origin is $G$-invariant,
the definition of $\mathrm{BMO}_{G,k}(\mathbb{R}^d)$ gives
$$
    \|f_1\|_{L^1(d\mu_k)}
    \le \mu_k(4B)\|f\|_{\mathrm{BMO}_{G,k}}
    <\infty.
$$
By Theorem \ref{main thm for maximal op} \ref{T star 1}, applied with $w=1$,
$\mathcal{T}^*f_1$ is finite $\mu_k$-almost everywhere.
As $\mu_k(E\cap B)>0$, we may choose $z\in E\cap B$
such that $\mathcal{T}^*f_1(z)<\infty$.
Since $T_{\mathbf N}c_B=0$ for every $\mathbf N$, we obtain
$$\mathcal{T}^*f_2(z)
    \leq \mathcal{T}^*f(z)+\mathcal{T}^*f_1(z)
    <\infty.$$

We next estimate the oscillation of $\mathcal{T}_{\mathbf N}f_2$
on $B$, uniformly in $\mathbf N$.
For $x,z\in B$ and $y\notin4B$, applying \eqref{smtness estimate last x changing} to compare $x$ and $z$ separately with the origin yields
$$
\begin{aligned}
    |K_{\mathbf N}(x,y)-K_{\mathbf N}(z,y)|
    &\le
    |K_{\mathbf N}(x,y)-K_{\mathbf N}(0,y)|
    +|K_{\mathbf N}(z,y)-K_{\mathbf N}(0,y)| \\
    &\lesssim
    \frac{R}{|y|\,\mu_k(B(0,|y|))}.
\end{aligned}
$$
Consequently, decomposing $(4B)^c$ into dyadic annuli
and using the doubling property, we obtain
$$
\begin{aligned}
    \sup_{\mathbf N}
    |\mathcal{T}_{\mathbf N}f_2(x)-\mathcal{T}_{\mathbf N}f_2(z)| 
    &\lesssim
    \sum_{j=2}^{\infty}
    \frac{2^{-j}}{\mu_k(2^{j+1}B)}
    \int_{2^{j+1}B}|f(y)-f_{4B}|\,d\mu_k(y).
\end{aligned}
$$
The BMO estimate and telescoping of the averages over
concentric balls give
$$\frac{1}{\mu_k(2^{j+1}B)}
    \int_{2^{j+1}B}|f(y)-f_{4B}|\,d\mu_k(y)
    \lesssim (j+1)\|f\|_{\mathrm{BMO}_{G,k}}.$$
It follows that
$$
\begin{aligned}
    \sup_{\mathbf N}
    |\mathcal{T}_{\mathbf N}f_2(x)-\mathcal{T}_{\mathbf N}f_2(z)|
    &\lesssim
    \sum_{j=2}^{\infty}2^{-j}(j+1)
    \|f\|_{\mathrm{BMO}_{G,k}} \\
    &\lesssim \|f\|_{\mathrm{BMO}_{G,k}}.
\end{aligned}
$$
Therefore, for all $x \in B$,
$$\mathcal{T}^*f_2(x)
    \le \mathcal{T}^*f_2(z)+C\|f\|_{\mathrm{BMO}_{G,k}}
    <\infty.$$
Together with the almost-everywhere finiteness of
$\mathcal{T}^*f_1$, this implies
$$
    \mathcal{T}^*f(x)
    \le \mathcal{T}^*f_1(x)+\mathcal{T}^*f_2(x)
    <\infty
    \quad\text{for $\mu_k$-almost every }x\in B.
$$
Taking $R=2^nR_0$, $n\ge0$, proves that $\mathcal{T}^*f$ is
finite $\mu_k$-almost everywhere on $\mathbb{R}^d$.

\par Again, for the boundedness from $ \mathrm{BMO}_{G,k}(\mathbb{R}^d)$ to $ \mathrm{BMO}_{k}(\mathbb{R}^d)$, we can proceed in a similar way as in Theorem \ref{bddness thm of partial sum} \ref{TN BMO}. Let us consider an arbitrary ball $B\subseteq \mathbb{R}^d$, and choose $x_1\in B$ such that $|\mathcal{T}^*f_2(x_1)|<\infty$, where we decompose $f$ as
\begin{eqnarray*}
    f& = & \big(f-f_{\mathcal{O}(B)}\big)\chi_{\mathcal{O}(B^*)} + \big(f-f_{\mathcal{O}(B)}\big)\chi_{\mathcal{O}(B^*)^c} + f_{\mathcal{O}(B)}\\
    & =: & f_1+ f_2 + f_3,
\end{eqnarray*}
and let $c_B=\mathcal{T}^*f_2(x_1)$. Then we can proceed as 
\begin{eqnarray*}
 && \frac{1}{\mu_k(B)}\int_B\big| \mathcal{T}^*f(x)- c_B\big|\, d\mu_k(x)\\ 
 &= & \frac{1}{\mu_k(B)}\int_B\big|\|\mathcal{T}_{\N}f(x)\|_{\ell^\infty(\mathbb{Z}^2)}-\|\mathcal{T}_{\N}f_2(x_1)\|_{\ell^\infty(\mathbb{Z}^2)}\big|\, d\mu_k(x)\\
 &\leq & \frac{1}{\mu_k(B)}\int_B\|\mathcal{T}_{\N}f(x)-\mathcal{T}_{\N}f_2(x_1)\|_{\ell^\infty(\mathbb{Z}^2)}\, d\mu_k(x)\\
&\leq & \frac{1}{\mu_k(B)}\int_B\|\mathcal{T}_{\N}f_1(x)\|_{\ell^\infty(\mathbb{Z}^2)}\, d\mu_k(x)+ \frac{1}{\mu_k(B)}\int_B\| \mathcal{T}_{\N}f_2(x)- \mathcal{T}_{\N}f_2(x_1)\|_{\ell^\infty(\mathbb{Z}^2)}\, d\mu_k(x) .  \end{eqnarray*}
Hence, the rest of the proof can be carried out in a similar manner as in the proof of Theorem \ref{bddness thm of partial sum} \ref{TN BMO} and is therefore omitted.
\end{proof}

\subsection{Proof of the pointwise convergence results}\label{subsec- proof of pointwise conv}
This section is devoted to proving the pointwise convergence results as an application of Theorem \ref{main thm for maximal op}. We provide the proof below.
\begin{proof}[Proof of Corollary \ref{main thm point wise conv}]
Let $\varphi\in C^\infty_c(\mathbb{R}^d)$. We will first show that $ \mathcal{T}_{\N}\varphi(x)$ converges for all $x \in \mathbb{R}^d$.  It is enough to prove that for $0<L<M$, the two sequences   
$$A_{L,\,M}=\sum\limits_{j=L}^{M} v_j \Big(e^{-a_{j+1}(-\Delta_k)^s}\varphi(x) - e^{-a_j(-\Delta_k)^s} \varphi(x)\Big)$$
$$\text{and } B_{L,\,M}=\sum\limits_{j=-M}^{-L} v_j \Big(e^{-a_{j+1}(-\Delta_k)^s}\varphi(x) - e^{-a_j(-\Delta_k)^s} \varphi(x)\Big)$$
converge to $0$ as $L, M\to \infty$.

\vspace{.5cm}
{\bf Convergence of $A_{L,\,M}$:}
\vspace{.5cm}

In this case, by using the bounds for the fractional heat kernel in Proposition \ref{esti for fractional Dunkl heat kernel} \ref{frac heat i}, we directly obtain
\begin{eqnarray*}
|A_{L,\,M}|&=&\Big|\sum\limits_{j=L}^{M} v_j \Big(e^{-a_{j+1}(-\Delta_k)^s}\varphi(x) - e^{-a_j(-\Delta_k)^s} \varphi(x)\Big) \Big|\\
&\leq & \|\{v_j\}\|_{\ell^\infty(\mathbb{Z})}\sum\limits_{j=L}^{M}  \int_{\mathbb{R}^d} \big|\tau^k_x h_{a_{j+1},s} (-y)- \tau^k_x h_{a_{j},s} (-y)\big|\,|\varphi(y)|\, d\mu_k(y) \\
&\lesssim & \sum\limits_{j=L}^{M}\Big(\frac{1}{\mu_k(B(x, a_{j+1}^{\frac{1}{2s}}))}+ \frac{1}{\mu_k(B(x, a_{j}^{\frac{1}{2s}}))}\Big)\int_{\mathbb{R}^d} |\varphi(y)|\, d\mu_k(y) \\
&\lesssim & \|\varphi\|_{L^1(d\mu_k)}\sum\limits_{j=L}^{M} a_j^{-\frac{d_k}{2s}}\\
&= & \|\varphi\|_{L^1(d\mu_k)}\, a_L^{-\frac{d_k}{2s}}\sum\limits_{j=L}^{M} \Big(\frac{a_L}{a_j}\Big)^{\frac{d_k}{2s}}\\
&\leq & \|\varphi\|_{L^1(d\mu_k)}\, a_L^{-\frac{d_k}{2s}}\sum\limits_{j=L}^{\infty} \ell^{-\frac{d_k}{2s}(-L+j)}\\
&\lesssim & \|\varphi\|_{L^1(d\mu_k)}\, a_L^{-\frac{d_k}{2s}},
\end{eqnarray*}
which goes to $0$ as $L, M\to \infty$.

\vspace{.5cm}
{\bf Convergence of $B_{L,\,M}$:}
\vspace{.5cm}

For  $B_{L,\,M}$, we use the fact that $$\int_{\mathbb{R}^d}h_{t,s} (y)\, d\mu_k(y)=1$$ 
for any $t>0$, and then split the expression into two parts in the following way.
\begin{eqnarray*}
|B_{L,\,M}|&=&\Big|\sum\limits_{j=-M}^{-L} v_j \int_{\mathbb{R}^d} \big( h_{a_{j+1},s} (y)-  h_{a_{j},s} (y)\big)\,\tau^k_x\varphi(y)\, d\mu_k(y) \Big| \\
&=&\Big|\sum\limits_{j=-M}^{-L} v_j \int_{\mathbb{R}^d} \big( h_{a_{j+1},s} (y)-  h_{a_{j},s} (y)\big)\,\Big(\tau^k_x\varphi(-y)-\varphi(x)\big)\, d\mu_k(y) \Big| \\
&\lesssim & \sum\limits_{j=-M}^{-L}  \int_{B(0,1)} \big| h_{a_{j+1},s} (y)-  h_{a_{j},s} (y)\big|\,\big|\tau^k_{-y}\varphi(x)-\varphi(x)\big|\, d\mu_k(y) \\
&& + \sum\limits_{j=-M}^{-L}  \int_{B(0,1)^c} \big| h_{a_{j+1},s} (y)-  h_{a_{j},s} (y)\big|\,\big|\tau^k_{-y}\varphi(x)-\varphi(x)\big|\, d\mu_k(y) \\
&=:&I +II.
\end{eqnarray*}

Using estimates for fractional Dunkl heat kernel in Proposition \ref{esti for fractional Dunkl heat kernel} \ref{frac heat i}, mean value theorem and the polar decomposition \eqref{polar dec}, we can write
\begin{eqnarray*}
    I &=&  \sum\limits_{j=-M}^{-L}  \int_{B(0,1)} \big| h_{a_{j+1},s} (y)-  h_{a_{j},s} (y)\big|\,\big|\tau^k_{-y}\varphi(x)-\varphi(x)\big|\, d\mu_k(y)\\
    &\leq & \sum\limits_{j=-M}^{-L}  \int_{B(0,1)} \Big\{ \frac{a_{j+1}}{(a_{j+1}^{\frac{1}{2s}}+|y|)^{d_k+2s}} + \frac{a_{j}}{(a_{j}^{\frac{1}{2s}}+|y|)^{d_k+2s}} \Big\} |y|\, \|\mathcal{F}_k(\nabla_k \varphi)\|_{L^1(d\mu_k)}\, d\mu_k(y)\\
     &\lesssim & \|\mathcal{F}_k(\nabla_k \varphi)\|_{L^1(d\mu_k)}\, \sum\limits_{j=-M}^{-L}  \int_{B(0,1)}  \frac{a_{j}\, |y|}{(a_{j}^{\frac{1}{2s}}+|y|)^{d_k+2s}} \, d\mu_k(y)\\
     &\lesssim & \|\mathcal{F}_k(\nabla_k \varphi)\|_{L^1(d\mu_k)}\, \sum\limits_{j=-M}^{-L}  a_j \int_0^1  \frac{t^{d_k}}{(a_{j}^{\frac{1}{2s}}+t)^{d_k+2s}} \, dt\\
     &\leq & \|\mathcal{F}_k(\nabla_k \varphi)\|_{L^1(d\mu_k)}\, \sum\limits_{j=-M}^{-L}  a^{\frac{1}{2s}}_j \int_0^{a^{-\frac{1}{2s}}_j}  \frac{1}{(1+t)^{2s}} \, dt\\
     &\leq & \|\mathcal{F}_k(\nabla_k \varphi)\|_{L^1(d\mu_k)}\, \sum\limits_{j=-M}^{-L}  a^{\frac{1}{2s}}_j \int_0^{a^{-\frac{1}{2s}}_j}  \frac{1}{(1+t)^{s}} \, dt\\
     &\leq & \|\mathcal{F}_k(\nabla_k \varphi)\|_{L^1(d\mu_k)}\, \sum\limits_{j=-M}^{-L}  a^{\frac{1}{2s}}_j \int_0^{a^{-\frac{1}{2s}}_j}  t^{-s} \, dt\\
      &\leq & \|\mathcal{F}_k(\nabla_k \varphi)\|_{L^1(d\mu_k)}\,a_{-L}^{1/2} \sum\limits_{j=-M}^{-L}  \Big(\frac{a_j}{a_{-L}}\Big)^{1/2} \\
      &\leq & \|\mathcal{F}_k(\nabla_k \varphi)\|_{L^1(d\mu_k)}\,a_{-L}^{1/2} \sum\limits_{j=-\infty}^{-L}  \Big(\frac{a_j}{a_{-L}}\Big)^{1/2} \\
      &\lesssim & \|\mathcal{F}_k(\nabla_k \varphi)\|_{L^1(d\mu_k)}\,a_{-L}^{1/2},
\end{eqnarray*}
which goes to $0$ as $L, M\to \infty$.
\par Similarly, combining the estimates for the fractional Dunkl heat kernel in Proposition \ref{esti for fractional Dunkl heat kernel} \ref{frac heat i} with the polar decomposition \eqref{polar dec}, we arrive at
\begin{eqnarray*}
 II &=&  \sum\limits_{j=-M}^{-L}  \int_{B(0,1)^c} \big| h_{a_{j+1},s} (y)-  h_{a_{j},s} (y)\big|\,\big|\tau^k_{-y}\varphi(x)-\varphi(x)\big|\, d\mu_k(y)\\
    &\leq & \sum\limits_{j=-M}^{-L}  \int_{B(0,1)^c} \Big\{ \frac{a_{j+1}}{(a_{j+1}^{\frac{1}{2s}}+|y|)^{d_k+2s}} + \frac{a_{j}}{(a_{j}^{\frac{1}{2s}}+|y|)^{d_k+2s}} \Big\}\, 2\|\mathcal{F}_k( \varphi)\|_{L^1(d\mu_k)}\, d\mu_k(y)\\ 
    &\lesssim & \|\mathcal{F}_k(\varphi)\|_{L^1(d\mu_k)}\, \sum\limits_{j=-M}^{-L}  a_j \int_1^\infty  \frac{t^{d_k-1}}{(a_{j}^{\frac{1}{2s}}+t)^{d_k+2s}} \, dt\\
    &\leq & \|\mathcal{F}_k( \varphi)\|_{L^1(d\mu_k)}\, \sum\limits_{j=-M}^{-L}   \int_1^\infty  \frac{1}{\big(1+{t}/{ a_{j}^{\frac{1}{2s}}}\big)^{2s}} \, \frac{dt}{t}\\
     &\leq & \|\mathcal{F}_k( \varphi)\|_{L^1(d\mu_k)}\, \sum\limits_{j=-M}^{-L}   \int_{a_{j}^{-\frac{1}{2s}}}^\infty  t^{-2s-1}\, dt\\
     &\leq & \|\mathcal{F}_k( \varphi)\|_{L^1(d\mu_k)}\,a_{-L} \sum\limits_{j=-M}^{-L}\frac{a_j}{a_{-L}}\\
     &\lesssim & \|\mathcal{F}_k(\varphi)\|_{L^1(d\mu_k)}\,a_{-L},
\end{eqnarray*}
which vanishes in the limit as $L, M\to \infty$.
\par Since the space of smooth functions with compact support is dense in $L^p(\mathbb{R}^d, d\mu_k)$ for $1\leq p<\infty$, Theorem \ref{main thm for maximal op} \ref{T star p} and \ref{T star 1} yield almost everywhere convergence for every function in $L^p(\mathbb{R}^d, d\mu_k)$. Likewise, because $L^p(\mathbb{R}^d, d\mu_k)\cap L^p(\mathbb{R}^d,w\,d\mu_k)$ is dense in $L^p(\mathbb{R}^d,w\,d\mu_k)$, the same almost everywhere convergence property remains valid for all functions in $L^p(\mathbb{R}^d,w\,d\mu_k)$ whenever $1\leq p<\infty$. Furthermore, for $1<p<\infty$, an application of the dominated convergence theorem yields convergence in the $L^p(\mathbb{R}^d,w\,d\mu_k)$ norm; and this also implies convergence in $w\,d\mu_k$-measure.

It remains to prove convergence in $w\,d\mu_k$-measure
when $p=1$. Let $w\in A_1^k$ be $G$-invariant, and let
$f\in L^1(\mathbb{R}^d,w\,d\mu_k)$. Denote the
almost-everywhere limit of $\mathcal{T}_{\mathbf N}f$ by $\mathcal{T}f$.
Given $\delta>0$, choose $\varphi\in C_c^\infty(\mathbb{R}^d)$
such that
$$\|f-\varphi\|_{L^1(w\,d\mu_k)}<\delta.$$
Since $w\in A_1^k\subset A_2^k$ and
$\varphi\in L^2(\mathbb{R}^d,w\,d\mu_k)$, part \ref{part 1 conv} gives $ \mathcal{T}_{\mathbf N}\varphi\longrightarrow \mathcal{T}\varphi$ in $L^2(\mathbb{R}^d,w\,d\mu_k)$. By linearity and the definition of $\mathcal{T}^*$, we have
$$
    |\mathcal{T}_{\mathbf N}f-\mathcal{T}f|
    \le |\mathcal{T}_{\mathbf N}\varphi-\mathcal{T}\varphi|
       +2\mathcal{T}^*(f-\varphi)
    \quad\text{almost everywhere}.
$$
Hence, for every $\varepsilon>0$, Chebyshev's inequality
and Theorem \ref{main thm for maximal op}\ref{T star 1} yield
$$
\begin{aligned}
    \int_{\{|\mathcal{T}_{\mathbf N}f-\mathcal{T}f|>\varepsilon\}}
        w(x)\,d\mu_k(x)
    &\le
    \int_{\{|\mathcal{T}_{\mathbf N}\varphi-\mathcal{T}\varphi|>\varepsilon/2\}}
        w(x)\,d\mu_k(x)+
    \int_{\{\mathcal{T}^*(f-\varphi)>\varepsilon/4\}}
        w(x)\,d\mu_k(x)\\
    &\le
    \frac{4}{\varepsilon^2}
    \|\mathcal{T}_{\mathbf N}\varphi-\mathcal{T}\varphi\|_{L^2(w\,d\mu_k)}^2
    +\frac{4C}{\varepsilon}
    \|f-\varphi\|_{L^1(w\,d\mu_k)}.
\end{aligned}
$$
Letting $\mathbf N\to(-\infty,\infty)$, we obtain
$$
    \limsup_{\mathbf N\to(-\infty,\infty)}
    \int_{\{|\mathcal{T}_{\mathbf N}f-\mathcal{T}f|>\varepsilon\}}
        w(x)\,d\mu_k(x)
    \le \frac{4C\delta}{\varepsilon}.
$$
Since $\delta>0$ is arbitrary, this proves that
$\mathcal{T}_{\mathbf N}f\to \mathcal{T}f$ in $w\,d\mu_k$-measure.
\end{proof}
\subsection{Proof of the local growth of the maximal operator}\label{subsec- proof of local growth}
We are now at the final part of the article, namely the proof of Theorem \ref{main thm local growth}, which is presented next.
\begin{proof}[Proof of Theorem \ref{main thm local growth}] The proof below does not require the auxiliary assumption
$a_{j+1}/a_j\leq\ell^2$ and applies directly to the given
sequences $\{a_j\}$ and $\{v_j\}$.

We first consider the case $1<p<\infty$. Let
$f(x)=f_1(x) + f_2(x)$, where we define
$$f_1=f_{\chi_{B(0,2r)}} \text{ and }f_2=f_{\chi_{B(0,1)-B(0,2r)}}.$$
Using H\"older's inequality and Theorem \ref{main thm for maximal op}\ref{T star p}, we obtain a bound
\begin{eqnarray}\label{one time use log}
 \frac{1}{\mu_k(B(0,r))}\int_{B(0,r)} \mathcal{T}^*f_1(x) \, d\mu_k(x) &\leq & \Big( \frac{1}{\mu_k(B(0,r))}\int_{B(0,r)} (\mathcal{T}^*f_1(x))^2  d\mu_k(x) \Big)^{1/2}  \\
 &\lesssim & \Big( \frac{1}{\mu_k(B(0,r))}\int_{B(0,2r)} |f_1(x)|^2  d\mu_k(x) \Big)^{1/2} \nonumber \\
 &\lesssim & \|f_1\|_{L^\infty(d\mu_k)}\leq \|f\|_{L^\infty(d\mu_k)}.\nonumber
\end{eqnarray}

Again, applying H\"older's inequality first on $\ell^p(\mathbb{Z})$ and then on $L^p(\mathbb{R}^d,d\mu_k)$, together with the bounds for the fractional Dunkl heat kernels Proposition \ref{esti for fractional Dunkl heat kernel} \ref{frac heat ii}, we obtain
\begin{eqnarray*}
 && \Big|\sum\limits_{j=N_1}^{N_2} v_j \Big(e^{-a_{j+1}(-\Delta_k)^s}f_2(x) - e^{-a_j(-\Delta_k)^s} f_2(x)\Big)  \Big|\\
 & \leq & \sum\limits_{j=N_1}^{N_2} \Big|v_j \int_{\mathbb{R}^d}\big(\tau^k_x h_{a_{j+1},s} (-y)- \tau^k_x h_{a_{j},s} (-y)\big)\, f_2(y)\, d\mu_k(y)\Big|\\
 & \leq &  \|\{v_j\}\|_{\ell^p(\mathbb{Z})} \Big[\sum\limits_{j=N_1}^{N_2}  \big(\int_{\mathbb{R}^d}\big|\tau^k_x h_{a_{j+1},s} (-y)- \tau^k_x h_{a_{j},s} (-y)\big|\, |f_2(y)|\, d\mu_k(y)\big)^{p'}
 \Big]^{1/p'}\\
  & \leq &  \|\{v_j\}\|_{\ell^p(\mathbb{Z})} \Big[\sum\limits_{j=N_1}^{N_2}  \Big(\int_{\mathbb{R}^d}\big|\tau^k_x h_{a_{j+1},s} (-y)- \tau^k_x h_{a_{j},s} (-y)\big|\, |f_2(y)|^{p'}\, d\mu_k(y)\Big)\\
  && \times \Big(\int_{\mathbb{R}^d}\big|\tau^k_x h_{a_{j+1},s} (-y)- \tau^k_x h_{a_{j},s} (-y)\big|\, d\mu_k(y)\Big)^{p'/p}
 \Big]^{1/p'}\\
 & \leq &  \|\{v_j\}\|_{\ell^p(\mathbb{Z})} \Big[\sum\limits_{j=N_1}^{N_2}  \Big(\int_{\mathbb{R}^d}\big|\tau^k_x h_{a_{j+1},s} (-y)- \tau^k_x h_{a_{j},s} (-y)\big|\, |f_2(y)|^{p'}\, d\mu_k(y)\Big)\\
  && \times \Big(\int_{\mathbb{R}^d}\tau^k_x h_{a_{j+1},s} (-y)\, d\mu_k(y)+ \int_{\mathbb{R}^d}\tau^k_x h_{a_{j},s} (-y)\, d\mu_k(y)\Big)^{p'/p}
 \Big]^{1/p'}\\
  & \leq & 2^{1/p}\, \|\{v_j\}\|_{\ell^p(\mathbb{Z})} \Big[ \int_{\mathbb{R}^d}\sum\limits_{j=-\infty}^{\infty} \big|\tau^k_x h_{a_{j+1},s} (-y)- \tau^k_x h_{a_{j},s} (-y)\big|\, |f_2(y)|^{p'}\, d\mu_k(y) \Big]^{1/p'}\\
  & \lesssim & \|\{v_j\}\|_{\ell^p(\mathbb{Z})} \Big[ \int_{\mathbb{R}^d}\int_{0}^{\infty}\big|\frac{\partial}{\partial t}\tau^k_x h_{t,s} (-y)\big|\,dt\, |f_2(y)|^{p'}\, d\mu_k(y) \Big]^{1/p'}\\
   & \lesssim &  \|\{v_j\}\|_{\ell^p(\mathbb{Z})} \Big[ \int_{\mathbb{R}^d}\int_{0}^{\infty}\frac{dt}{\big(d_G(x,y) + t^{\frac{1}{2s}}\big)^{d_k+2s}}\, |f_2(y)|^{p'}\, d\mu_k(y) \Big]^{1/p'}\\
   & \lesssim &  \|\{v_j\}\|_{\ell^p(\mathbb{Z})} \Big[ \int_{\mathbb{R}^d} \frac{|f_2(y)|^{p'}}{\big( d_G(x,y)\big)^{d_k} }\, d\mu_k(y) \Big]^{1/p'}.
\end{eqnarray*}
Now, the last estimate and the polar decomposition \eqref{polar dec} yield
\begin{eqnarray}\label{one time use log2}
 && \frac{1}{\mu_k(B(0,r))}\int_{B(0,r)} \mathcal{T}^*f_2(x) \, d\mu_k(x)\\
 &\lesssim &   \frac{\|\{v_j\}\|_{\ell^p(\mathbb{Z})}}{\mu_k(B(0,r))} \int_{B(0,r)}\Big[ \int_{\mathbb{R}^d} \frac{|f_2(y)|^{p'}}{\big( d_G(x,y)\big)^{d_k} }\, d\mu_k(y) \Big]^{1/p'}d\mu_k(x) \nonumber\\
  &\leq &   \frac{\|\{v_j\}\|_{\ell^p(\mathbb{Z})}\|f_2\|_{L^\infty(d\mu_k)}}{\mu_k(B(0,r))} \int_{B(0,r)}\Big[ \int\limits_{2r<|y|<1} \frac{1}{\big( d_G(x,y)\big)^{d_k} }\, d\mu_k(y) \Big]^{1/p'}d\mu_k(x) \nonumber\\
  &\leq &   \frac{\|\{v_j\}\|_{\ell^p(\mathbb{Z})}\|f_2\|_{L^\infty(d\mu_k)}}{\mu_k(B(0,r))} \int_{B(0,r)}\Big[ \int\limits_{2r<|y|<1} \frac{2^{d_k}}{|y|^{d_k} }\, d\mu_k(y) \Big]^{1/p'}d\mu_k(x) \nonumber\\
  &\lesssim &   \frac{\|\{v_j\}\|_{\ell^p(\mathbb{Z})}\|f_2\|_{L^\infty(d\mu_k)}}{\mu_k(B(0,r))} \int_{B(0,r)}\Big[ \int\limits_{r<|y|<2} \frac{1}{|y|^{d_k} }\, d\mu_k(y) \Big]^{1/p'}d\mu_k(x) \nonumber\\
  &\lesssim &   \frac{\|\{v_j\}\|_{\ell^p(\mathbb{Z})}\|f_2\|_{L^\infty(d\mu_k)}}{\mu_k(B(0,r))} \int_{B(0,r)}\Big[ \int_r^2 \frac{dt}{t} \Big]^{1/p'}d\mu_k(x) \nonumber\\
  &\lesssim & \|\{v_j\}\|_{\ell^p(\mathbb{Z})}\Big(\log \frac{2}{r}\Big)^{1/p'}\|f\|_{L^\infty(d\mu_k)}.\nonumber
\end{eqnarray}
Since $\log \frac{2}{r}>1$, it follows from \eqref{one time use log} and \eqref{one time use log2} that 
\begin{eqnarray*}
    \frac{1}{\mu_k(B(0,r))}\int_{B(0,r)} \mathcal{T}^*f(x) & \leq &  \frac{1}{\mu_k(B(0,r))}\int_{B(0,r)} \mathcal{T}^*f_1(x) + \frac{1}{\mu_k(B(0,r))}\int_{B(0,r)} \mathcal{T}^*f_2(x)\\
    &\lesssim & \Big(\log \frac{2}{r}\Big)^{1/p'}\|f\|_{L^\infty(d\mu_k)}. 
\end{eqnarray*}
The proofs for the other two cases, namely $p=1$ and $p=\infty$, are similar, and in fact simpler than the present case; therefore, they are not included in this article.
\end{proof}

\subsection*{Acknowledgments} 
The author is supported by Institute Postdoctoral Fellowship from IIT Bombay.

\subsection*{Data availability}
Data availability is not applicable.

\subsection*{Competing interests}
The author declares that he has no competing interests.

\bibliographystyle{abbrv}

\bibliography{biblio}

\end{document}